%% file: main.tex
\documentclass{article}
\RequirePackage{etex}

\usepackage[margin=1in]{geometry}
\usepackage[english]{babel}
\usepackage[utf8]{inputenc}
\usepackage{subcaption}
\usepackage{amsmath}
\usepackage{csquotes}
\usepackage{amssymb}
\usepackage{amsfonts}
\usepackage{amsthm}
\usepackage{mathrsfs}
\usepackage[all]{xy}
\usepackage[pdftex]{graphicx}
\usepackage{color}
\usepackage{url}
\usepackage{indent first}
\usepackage[labelfont=bf,labelsep=period,justification=raggedright]{caption}
\usepackage[colorinlistoftodos]{todonotes}
\usepackage{tkz-fct}
\usepackage{tikz}
\usetikzlibrary{calc}
\usepackage{multicol}
\PassOptionsToPackage{dvipsnames,svgnames}{xcolor}
\usepackage{textcomp}

\DeclareMathOperator{\Hom}{Hom}

\DeclareMathOperator{\Spec}{Spec}
\DeclareMathOperator{\Proj}{Proj}

\newcommand{\CC}{\mathbb{C}}

\newcommand{\FF}{\mathbb{F}}
\newcommand{\GG}{\mathbb{G}}

\newcommand{\NN}{\mathbb{N}}

\newcommand{\PP}{\mathbb{P}}

\newcommand{\RR}{\mathbb{R}}
\newcommand{\VV}{\mathbb{V}}
\newcommand{\ZZ}{\mathbb{Z}}

\newcommand{\mcB}{\mathcal{B}}
\newcommand{\mcE}{\mathcal{E}}
\newcommand{\mcF}{\mathcal{F}}
\newcommand{\mcG}{\mathcal{G}}

\newcommand{\mcI}{\mathcal{I}}
\newcommand{\mcL}{\mathcal{L}}

\newcommand{\mcN}{\mathcal{N}}
\newcommand{\mcO}{\mathcal{O}}

\newcommand{\mcU}{\mathcal{U}}

\newcommand{\mfm}{\mathfrak{m}}

\def\multiset#1#2{\ensuremath{\left(\kern-.3em\left(\genfrac{}{}{0pt}{}{#1}{#2}\right)\kern-.3em\right)}}

\theoremstyle{plain}
\newtheorem{thm}{Theorem}
\newtheorem{cor}[thm]{Corollary}

\newtheorem{prop}[thm]{Proposition}
\newtheorem{lemma}[thm]{Lemma}

\theoremstyle{definition}
\newtheorem{defn}[thm]{Definition}

\theoremstyle{remark}
\newtheorem{rem}[thm]{Remark}
\newtheorem*{ex}{Example}

\numberwithin{equation}{section}
\numberwithin{thm}{section}

\usepackage{arxiv}

\usepackage[utf8]{inputenc} 
\usepackage[T1]{fontenc}    
\usepackage{hyperref}       
\usepackage{url}            
\usepackage{booktabs}       
\usepackage{amsfonts}       
\usepackage{nicefrac}       
\usepackage{microtype}      
\usepackage{lipsum}
\usepackage{titlesec}

\titleformat{\subsection}[runin]
{\normalfont\normalsize\bfseries}{\thesubsection}{1em}{}
\titleformat{\subsubsection}[runin]
{\normalfont\normalsize\bfseries}{\thesubsubsection}{1em}{}

\title{On the arithmetic Hilbert-Samuel theorem : a proof by deformation}

\usepackage[style=alphabetic]{biblatex}
\author{Dorian Ni}

\begin{document}

\maketitle

\input{1-Introduction/abstract}

\section{Introduction}
\input{1-Introduction/Introduction}

\input{2-AnalyticTube/AnalyticTube}

\section{Bornological invariants in Arakelov geometry and the arithmetic Hilbert-Samuel theorem}\label{partiebornologies}
\input{3-Bornological/Bornology}

\section{Proof of the arithmetic Hilbert-Samuel theorem}

\input{5-Proof/introproof}
\input{5-Proof/GeometricHilbertSamuel}

\input{ProofArithmeticHilbertSamuel/Keytechnicallemma}

\input{ProofArithmeticHilbertSamuel/TheCaseOfPn}
\input{ProofArithmeticHilbertSamuel/TheProof}

\section*{\textsc{Appendix}}

\input{3-Bornological/Hilbertinvariants}
\input{3-Bornological/statementOfTheHilbertSamuelTheorem}

\section{The deformation to the projective completion of the cone : case of a divisor}

\input{4-Deformation/introDeformation}

\input{4-Deformation/algebraicdeformation}

\input{4-Deformation/hilbertsamuelgeometric}
\input{4-Deformation/deformationOfAnalyticsTubes}

\input{4-Deformation/invarianceHfunctionByDeformation}




\nocite{*}
\printbibliography

\end{document}

%% file: 1-Introduction/abstract.tex
\begin{abstract}
We give a new proof the arithmetic Hilbert-Samuel theorem by using classical reductions in the theory of coherent sheaves, a direct proof in the case of the projective space and the conservation of some numerical invariants, called arithmetic Hilbert invariants, through the deformation to the projective completion of the cone. This construction lies at the intersection of deformation theory and Arakelov geometry. It provides a deformation of a Hermitian line bundle over the deformation to the normal cone. 
\end{abstract}


%% file: 1-Introduction/Introduction.tex
\subsection{Arithmetic Hilbert-Samuel theorem.} The classical version of the arithmetic Hilbert-Samuel theorem describes
the asymptotic behavior of $\overline{\mcL}^{\otimes n}$ for $n \gg 0$, where $\overline{\mcL}$ is a semipositive Hermitian line bundle, by using the associated $L^2$-norms:\\
Let $X$ be a projective scheme over $\Spec \ZZ$ such that $X_\CC$ is smooth, let $|\cdot|$ be a Hermitian metric on a line bundle $\mcL$. Then a Riemannian continuous metric defines a measure $d\mu$ on the complex manifold $X(\CC)$, and we can define
    $$||s||^2_{L^2} =\int_{X(\CC)}|s(x)|_x^2d\mu(x)$$
for $s\in H^0(X,\mcL^{\otimes n})\otimes_\ZZ \CC$.\\
The arithmetic Hilbert-Samuel theorem then states the following:
\begin{thm}
Let $X$ be a flat projective scheme over $\Spec \ZZ$ such that $X_\CC$ is smooth. Let $\overline{\mcL}$ be a semipositive Hermitian line bundle over $X$. Fix a Riemannian continuous metric invariant by complex conjugation on $X(\CC)$, and denote $||\cdot||_{L^2}$ the associated $L^2$-norms on $H^0(X(\CC), \mcL^{\otimes n})$. Let $d$ the dimension of $X$. Then,
$$\frac{d!}{n^d}\widehat{\chi}(H^0(X,\mcL^{\otimes n}), ||\cdot||_{L^2}) \text{ converges towards a finite limit when } n \to \infty.$$
\end{thm}
This theorem has first been proven by Gillet and Soulé in \cite{Gillet1992} by using arithmetic intersection theory in the case of flat projective scheme over $\Spec \ZZ$ with smooth generic fiber and of smooth metrics on line bundles.
Zhang, in \cite{zhang}, then extended this theorem to flat reduced projective scheme over $\Spec \ZZ$ by using a resolution of singularities theorem. Finally, Randriam in \cite{randriam} generalized to continuous metric by using the extended definition of the arithmetic intersection, as defined in \cite{Maillot2000}.\\
Thus, the arithmetic Hilbert-Samuel theorem which first arose as a consequence of the arithmetic Riemann-Roch theorem still holds outside the setting of arithmetic intersection theory. Similarly, even though the classical Hilbert-Samuel theorem can be seen as a consequence of the intersection theory, through the Hirzebruch-Riemann-Roch theorem, see \cite{fulton}, it also holds outside the setting of intersection theory. Furthermore, it was proven through elementary means before the development of intersection theory, in search for solid grounds for intersection theory.\\
In \cite{Abbes1995}, Abbes and Bouche managed to show that the arithmetic Hilbert-Samuel could also be proven without using intersection theory. By using $L^2$-methods, they adapted the elementary "dévissage" technique to give a proof of the theorem in the case where the generic fiber is smooth and the norms as well. In their paper, they also hinted that there exists a $L^\infty$-theoretic version of the arithmetic Hilbert-Samuel theorem.\\
Using the $L^\infty$ setting, \cite{rumelyvauvarley} manages to adapt the monomial basis of the Gröbner theory to prove the classical Hilbert-Samuel theorem. More recently, developing on the theory of Okounkov bodies, Lazarsfeld and Mustata, in \cite{lazarsfeld2008convex}, and Kaveh and Khovanskii, in \cite{kaveh2008convex}, proved various extensions of the classical Hilbert-Samuel theorem. Yuan, in \cite{yuan2009volumes}, and Boucksom and Chen, in \cite{Boucksom_2011}, managed to adapt their techniques to give proofs of the arithmetic Hilbert-Samuel theorem in the reduced case and with continuous norms on line bundles.
The last proofs also used the $L^\infty$ setting, therefore arising in the case of continuous norms on line bundles without any positivity assumption on the metrics.\\

The proof given in this paper will be borrowing properties of the $L^\infty$ and the $L^2$ settings via the theory of analytic tubes and the canonical topological structures on the total spaces of sections, as developed in \cite{Bost-Charles}. In particular, using this canonical structures allowed a conservation of number theorem in \cite{NI2022}, where the conservation of the arithmetic Hilbert invariants along the deformation of the projective completion of the cone is proven.\\

Working in this setting will not only give a geometrical proof of the arithmetic Hilbert-Samuel, but this proof arises naturally in the context of possibly non-reduced projective schemes over $\Spec \ZZ$ and upper semi-continuous semi-norms on line bundles.  

\subsection{Geometry of the tube and arithmetic Grauert tubes.} In \cite{grauert}, Grauert initiated a new paradigm concerning the notion of ampleness of a line. Grauert's ampleness criterion states that the ampleness of a line bundle is characterised by the fact that the total space of $\mcL^{\vee}$ denoted $\VV_X(\mcL)$ is a modification of an affine scheme, see in particular the version given in EGA II \cite[Chapter 2, Sections 8.8 to 8.10]{grothendieck}.\\
In Arakelov geometry, this shift from the properties of the line bundle to the properties of the total space $\VV_X(\mcL)$ also claims interesting results. Indeed, considering the total space of an ample Hermitian line bundle $\overline{\mcL}$ over a scheme $X$ projective over $\Spec \ZZ$ allows to obtain geometric insight on the metric on $\mcL$. This is done by considering the corresponding \emph{analytic tube}.\\ As $\mcL$ is ample, by \cite{grauert}, the analytic space $$\VV_{X(\CC)}(\mcL) = \{(x,\varphi)\;|\; x\in X(\CC), \varphi \in \mcL^\vee_{x}\} $$ is a modification of a Stein space, and the analytic tube associated to the metric on $\mcL$ is then: $$T = \{(x,\varphi)\;|\; |\varphi|_x\leq 1 \}$$ It is a holomorphically convex compact of $\VV_{X(\CC)}(\mcL)$.\\ Considering the analytic tube allows \cite{bostDwork} and \cite{randriam} to derive some lifting results, similar to the Hörmander's $L^2$ estimates  (see for example \cite{Manivel1993}) but without any smoothness assumption on $X$. Furthermore, in the formalism of $A$-schemes developed in \cite{Bost-Charles}, an arithmetic Grauert's ampleness criterion is given. In \cite{Bost-Charles}, the semipositivity of a seminormed line bundle $\overline{\mcL}$ is then characterized by the properties of the topological spaces $H^0(\VV_X(\overline{\mcL}), p_X^*\mcF)$ where $p_X^*\mcF$ is the pullback to $\VV_X(\mcL)$ of a coherent sheaf $\mcF$ on $X$.\\

Building on the theory of the analytic tube, we will work with the canonical topological structure induced by the morphism:
$$H^0(\VV_X(\mcL), p_X^*\mcF) \to H^0(T,p_X^*\mcF) $$
on the total space of sections: 
$$H^0(\VV_X(\mcL), p_X^*\mcF) = \bigoplus_{n\in \NN} H^0(X,\mcL^{\otimes n}\otimes \mcF)$$
Consequently, the natural framework for this article will be possibly non-reduced projective schemes over $\Spec \ZZ$ and upper semicontinuous seminorms on line bundles, without any additional assumptions of smoothness on the space or on the metric.

\subsection{Arithmetic Hilbert invariants.} As suggested in the previous paragraph, one important asset of using the total space of $\mcL$, is that, if $\mcF$ is a coherent sheaf, considering the analytic tube in $\VV_X(\mcL)$ endows the space $H^0(\VV_X(\mcL),p_X^*\mcF)=\bigoplus_{n\in \NN} H^0(X,\mcL^{n}\otimes \mcF)$ with a \emph{canonical} topological structure. Furthermore, under the condition that $\overline{\mcL}$ is a Hermitian line bundle, that the underlying space $X$ is reduced and that $\mcF$ is a Hermitian vector bundle, this topological structure can be understood via the $L^p$-norms, see \cite{Bost-Charles}.\\
In \cite{NI2022}, this topological structure allows to define numerical invariants $\overline{c}_r(X,\overline{\mcL}, \mcF)$ and $\underline{c}_r(X,\overline{\mcL},\mcF)$ in $\RR\cup\{ +\infty \}$, inspired by the arithmetic intersection theory and the arithmetic Hilbert-Samuel theorem, that verify several important properties : 
\begin{itemize}
    \item It satisfies a projection formula for closed immersions.
    \item It is additive with respect to exact sequences of coherent sheaves.
    \item Assuming the results of arithmetic intersection theory: if $\overline{\mcL}$ is an ample Hermitian line bundle on a projective scheme $X$ over $\Spec \ZZ$ of dimension $d$ with smooth generic fiber, then:
        $$\overline{c}_d(X,\overline{\mcL}, \mcO_X)= \underline{c}_d(X,\overline{\mcL},\mcO_X)= \widehat{c_1}(\overline{\mcL})^d $$
\end{itemize}

\subsection{Results.} The introduction of the arithmetic Hilbert invariants allows the following interpretation of the arithmetic Hilbert-Samuel theorem, which will be proven in paragraph \ref{Proooooooof}: 

\begin{thm}\label{thprincipal} Let $X$ be a projective scheme over $\Spec \ZZ$ and $\overline{\mcL}$ a uniformly definite semipositive seminormed line bundle over $X$. Let $\mcF$ be a coherent sheaf over $X$. Let $d$ be the dimension of the support of $\mcF$.  Then, 
$$-\infty < \underline{c}_{d}(X,\overline{\mcL}, \mcF) = \overline{c}_{d}(X,\overline{\mcL},\mcF)<+\infty$$
Furthermore, if $\overline{\mcL}$ is ample, we have 
$$0 < \underline{c}_{d}(X,\overline{\mcL}, \mcF) = \overline{c}_{d}(X,\overline{\mcL},\mcF)<+\infty$$
\end{thm}

This implies classical statements of the Hilbert-Samuel theorem, and, in particular, in the case of $\mcF=\mcO_X$, we recover the fact the semipositivity condition is not needed, as in \cite{rumelyvauvarley}, \cite{yuan2009volumes}, \cite{Boucksom_2011}.

\begin{cor}Let $X$ be a projective scheme over $\Spec \ZZ$ of dimension $d$ and $\overline{\mcL}$ a uniformly definite normed line bundle over $X$ such that the underlying line bundle $\mcL$ is ample over $\Spec \ZZ$, then 
$$-\infty < \underline{c}_{d}(X,\overline{\mcL}, \mcO_X) = \overline{c}_{d}(X,\overline{\mcL},\mcO_X)<+\infty$$
\end{cor}

We will furthermore prove a stronger result in paragraph \ref{CaseofPnO1}, in the case of $\PP^l_\ZZ$ and $\mcO_\PP(1)$, without the uniformly definite condition that ensures finiteness.

\begin{thm}\label{CasSupplementairedePn} Let $X = \PP^l_\ZZ$. Consider a seminormed line bundle structure on $\mcO_\PP(1)$. \\
Then, 
$$\underline{c}_{l+1}(\PP^l_\ZZ,\overline{\mcO_\PP(1)}, \mcO_\PP) = \overline{c}_{l+1}(\PP^l_\ZZ,\overline{\mcO_\PP(1)},\mcO_\PP)\in \RR\cup \{+\infty\}$$
\end{thm}

\subsection{Strategy of the proof.} Firstly, the compatibility of arithmetic Hilbert invariants with exact sequences of coherent sheaves will allow us to use classical reductions in the theory of coherent sheaves. Then, we will use the deformation to the projective completion of the cone introduced in \cite{NI2022}, to reduce ourselves to the standard case of $\PP^l_\ZZ$ and $\mcO_\PP(1)$ endowed with a metric stable by the action of $U(1)^{l+1}$, where $U(1)$ is the group of the complex numbers of module $1$ and where $U(1)^{l+1}$ acts by multiplication on each variable of $\PP^l_\ZZ$. Then, building on Fekete superaddivity lemma and using the algebra structure of the topological space $\bigoplus_n H^0(\PP^l_\ZZ, \mcO(n))$, we will give a direct proof of the standard case.

\subsection{Outline of the paper.}
The formalism of the geometry of the analytic tube associated to a seminormed line bundle will be detailed in section 2. In particular, we will describe the extension of the notion of semipositivity to seminormed line bundles. In section 3, we will introduce the natural bornological structure on the total space of sections and recall the fact that the bornology is an invariant of the equilibrium seminorm associated to the given seminorm. \\
Finally, in section 4, we will give a proof of the arithmetic Hilbert-Samuel theorem which relies only on the deformation to the projective completion of the cone, on classical reductions and on the case of $\PP^l_\ZZ$ and the line bundle $\mcO_\PP(1)$.\\
In appendix, in section 5 and 6, we extracted from \cite{NI2022} the facts needed for our proof. More precisely, in section 5, we will give the definition of the arithmetic Hilbert invariants and recall some properties proven in \cite{NI2022}. In section 6, we will recall the main properties of the deformation to the projective completion of the cone introduced in \cite{NI2022}, including a description of the natural transverse action of $\GG_m$ and the invariance by deformation theorem for arithmetic Hilbert invariants.\\

\subsection{Acknowledgements.}
I would like to thanks François Charles for sharing its views on Arakelov geometry with me, for the numerous discussions and for allowing me to dive into his joint work with Jean-Benoît Bost to apprehend the subject. This paper owes him a great intellectual debt.\\
This project has received funding from the European Research Council (ERC) under the European Union’s Horizon 2020 research and innovation programme (grant agreement No 715747).

%% file: 2-AnalyticTube/AnalyticTube.tex
\section{Seminormed vector bundles and the associated analytic tube}

We first give a few definitions in the usual setting of Arakelov geometry and then give their counterpart in the framework of \cite{Bost-Charles}.
\subsection{Seminormed vector bundles\\}

\begin{defn} Let $X$ be a projective scheme over $\Spec \ZZ$. 
A \emph{seminorm} $|\cdot|$ on a locally free sheaf of finite rank $\mcF$ on $X$, is, for each reduced point $x$ of $ X_{\CC}$, a seminorm $|\cdot|_x$ on the complex vector space $\mcF_{|x} = \mcF_x/\mfm_x\mcF_x$.\\
Furthermore, the seminorm $|\cdot|$ is said to be \emph{upper semicontinuous} (resp. \emph{continuous}) if for every open set $\mcU$ of the analytic space $X_{\CC}$, and every $s\in H^0(\mcU,\mcF)$,  the map 
$$\begin{array}{ccccc}
|s| & : & \mcU & \to & \RR^{+} \\
 & & x & \mapsto & |s(x)|_x \\
\end{array}$$
is upper semicontinuous (resp. continuous).\\
A \emph{seminormed vector bundle} $\overline{\mcF} = (\mcF,|\cdot|)$ on $X$, is a locally free sheaf of finite rank $\mcF$ endowed with an upper semicontinuous seminorm invariant by complex conjugation.\\
\end{defn}

\begin{rem} For more details on complex conjugation, see \cite[Remark 2.2]{NI2022}.
\end{rem}

\begin{defn} Let $X$ be a projective scheme over $\Spec \ZZ$. A seminorm $|\cdot|$ on a seminormed vector bundle $\overline{\mcF}$ is said to be \emph{definite} if all the seminorms $|\cdot|_x$ are definite. $\overline{\mcL}$ is then called a \emph{normed vector bundle} and $|\cdot|$ a \emph{norm on $\mcL$}.\\
In the special case where $|\cdot|$ is definite continuous and Hermitian,  $\overline{\mcF}$ is usually called a \emph{Hermitian vector bundle} and $|\cdot|$ a continuous Hermitian metric.\\
\end{defn}

In the case of an upper semicontinuous seminorm on a vector bundle $\mcF$, we will need a more restrictive notion than the notion of definite seminorms:\\

\begin{defn}\label{uniformlydef}
Let $X$ be a projective scheme over $\Spec \ZZ$. A norm $|\cdot|$ on a normed vector bundle $\overline{\mcF}$ is said to be \emph{uniformly definite} if there exists a continuous norm $|\cdot|'$ on $\mcF$, such that $|\cdot|\geq |\cdot|'$.\\
\end{defn}

The analog of the notion of ample line bundle is given in \cite{zhang} and we also use the associated notion of semipositivity:

\begin{defn}\label{amplehermitian} Let $X$ be a projective scheme over $\Spec \ZZ$. A Hermitian line bundle $\overline{\mcL}$ is \emph{semipositive} if :
\begin{itemize}
    \item[i)] $\mcL$ is ample.
    \item[ii)] for every open set $\mcU$ of the analytic space $X_{\CC}$, and every non-vanishing section $s\in H^0(\mcU,\mcL)$,  the map 
$$\begin{array}{ccccc}
-\log|s|^2 & : & \mcU & \to & \RR  \\
 & & x & \mapsto & -\log|s(x)|^2_x \\
\end{array}$$ is plurisubharmonic.
\end{itemize}
Furthermore, a Hermitian line bundle $\overline{\mcL}$ over $X$ is \emph{ample} if it is semipositive and for large enough $n$, there is a basis of $H^0(X, \mcL^{\otimes n})$, consisting of strictly effective sections, that is sections $s\in H^0(X, \mcL^{\otimes n})$ such that $|s|_x < 1$ for all $x \in X(\CC)$.
\end{defn}

This definition will be extended in definition \ref{definitionAmpleLineBundleHolomorphicallyConvex} to seminormed line bundle.\\

\begin{ex}
In the case of the projective space $\PP^l_\ZZ$, the line bundle $\mcO_\PP(1)$ can be endowed with the Fubini-Study metric $|\cdot|_{FS}$ :\\
Let $[x_0:\dots : x_l] \in \PP^l_\CC$, then 
$\mcO_\PP(1)_{[x_0:\dots : x_l]}$ is isomorphic to $\Hom_\CC(\CC(x_0,\dots,x_l),\CC)$. 
If $P\in \Hom_\CC(\CC(x_0,\dots,x_l),\CC)$, we define $${|P|_{FS}}_{[x_0:\dots : x_l]} :=  \frac{|P(x_0,\dots,x_l)|}{(|x_0|^2+\dots + |x_l|^2)^{1/2}}$$
This defines a semipositive Hermitian line bundle $\overline{\mcO_\PP(1)} = (\mcO_\PP(1),|\cdot|_{FS} )$.\\
Furthermore, for all $\epsilon > 0$, $(\mcO_\PP(1),e^{-\epsilon}|\cdot|_{FS} )$ is an ample Hermitian line bundle.\\
\end{ex}

The following definition of pullback of seminormed vector bundle give further examples:\\

\begin{defn}
Let $f: X \to Y$ be a morphism between projective schemes over $\Spec \ZZ$. Then the \emph{pullback} $f^*\overline{\mcF}$ of a seminormed vector bundle $\overline{\mcF}= (\mcF, |\cdot|)$ over $Y$ is the seminormed vector bundle of underlying sheaf $f^*\mcF$ and such that the seminorm $|\cdot|'$ over $f^*\mcF$ verifies that the pullback map $\mcF_{|f(x)} \to (f^*\mcF)_{|x} $ is an isometry. 
\end{defn}

\begin{ex}
Let $X$ be a projective scheme over $\Spec \ZZ$, $\mcL$ be a line bundle over $X$ ample over $\Spec \ZZ$.
Now, let $\pi : X \to \PP^l_\ZZ$ be a closed immersion induced by global sections of $\mcL^{\otimes n}$, then $\pi^*(\mcO_\PP(1),|\cdot|_{FS})$ endows $\mcL^{\otimes n}$ with a structure of semipositive Hermitian line bundle. Taking the $n$-root of the seminorm, we get a structure of semipositive Hermitian line bundle on $\mcL$.\\
Furthermore, any semipositive Hermitian metric on an ample line bundle can be approximated uniformly by roots of pullbacks of Fubini-Study metrics, see for example \cite{tian}, \cite{BoucheB1}, \cite{berman2007bergman}.\\

\end{ex}

\subsection{The analytic tube\\[2mm]}\label{analytictube}
Grauert proved that the ampleness of a line bundle $\mcL$ over $X$ a projective scheme over $\Spec \ZZ$ can be characterized by the fact that $\VV_X(\mcL)$ is a modification over $\Spec \ZZ$ of an affine scheme, where $\VV_X(\mcL)$ is the affine scheme over $X$ defined by the sheaf of algebra $\bigoplus_{n\in \NN} \mcL^{\otimes n}$, see EGA II \cite[Chapter 2, Sections 8.8 to 8.10]{grothendieck}.\\
In a similar manner, \cite{Bost-Charles} gives a definition of the \emph{analytic tube} associated to a seminormed line bundle $\overline{\mcL}$, which allows an corresponding definition of the semipositivity of a seminormed line bundle $\overline{\mcL}$ as a property of the analytic tube.\\

More precisely, the complex points of $\VV_X(\mcL)$ can be written as $$\VV_X(\mcL)(\CC) = \{(x,\varphi)\;|\; x\in X(\CC), \varphi \in \mcL^\vee_{x}\}$$
as points of the total space of the dual line bundle $\mcL^\vee$.\\
If $|\cdot|$ is an upper semicontinuous seminorm on $\mcL$, then the dual pseudo-norm, also denoted $|\cdot|$, is lower semicontinuous and definite, thus $$T = \{(x,\varphi)\;|\; x\in X(\CC), \varphi \in \mcL^\vee_{|x}, |\varphi|\leq 1\}$$ is compact.
\begin{defn}
With the above notation, $T$ is called the \emph{analytic tube associated to $\overline{\mcL}$}.
\end{defn}

\begin{prop}
The datum of $(\VV_X(\mcL),T)$ where $T$ is a compact subset of $\VV_X(\mcL)(\CC)$ containing the zero section $0_X \subset \VV_X(\mcL)(\CC)$ stable by multiplication by any complex number $\lambda$ such that $|\lambda| \leq 1$ and invariant by complex conjugation is equivalent to the datum of a seminormed line bundle.
\end{prop}
\begin{proof}
See \cite{Bost-Charles}.
\end{proof}

Later on, we will use the terminology \emph{seminormed line bundle} for both notions.\\

\begin{ex}

The total space associated to $(\PP^N_\ZZ, \mcO_\PP(1))$ is the following: $$\VV_{\PP^N_{\ZZ}}(\mcO_\PP(1))(\CC) = \{([x_0:\dots : x_N],(y_0,\dots,y_N)) \;|\; [x_0:\dots : x_N]\in \PP^N(\CC),(y_0,\dots,y_N)\in \CC (x_0,\dots,x_N) \} $$
And then, the tube associated to the Fubiny-Study metric is
$$T = \{ ([x_0:\dots : x_N],(y_0,\dots,y_N)) \in \VV_{\PP^N_{\ZZ}}(\mcO_\PP(1))(\CC) \;|\; \sum |y_i|^2 \leq 1 \}  $$
\end{ex}

In particular, the notion of uniformly line bundles definite admits a convenient translation in this setting: 

\begin{prop}
The seminormed line bundle $\overline{\mcL} = (\VV_X(\mcL),T)$ is uniformly definite if, and only if, $T$ contains a neighborhood of the zero section $0_X$.
\end{prop}

See \cite[paragraph 2.2]{NI2022} for an example of definite but not uniformly definite seminorm on a line bundle.\\

In this context, we can now extend the definition of a semipositive Hermitian line bundle.

\begin{defn}\label{definitionAmpleLineBundleHolomorphicallyConvex}
Let $X$ be a projective scheme over $\Spec \ZZ$. A seminormed line bundle $\overline{\mcL} = (\VV_X(\mcL),T)$ over $X$ is said to be \emph{semipositive} if $\VV_X(\mcL)$ is a modification over $\Spec \ZZ$ of an affine scheme and if $T$ is a holomorphically convex compact in $\VV_X(\mcL)(\CC)$. \\
Furthermore, a seminormed line bundle $\overline{\mcL}$ over $X$ is \emph{ample} if it is semipositive and for large enough $n$, there is a basis of $H^0(X, \mcL^{\otimes n})$ consisting of strictly effective sections.

\end{defn}

To see that this extends definition \ref{amplehermitian}, see \cite{Bost-Charles}, EGA II \cite[Chapter 2, Sections 8.8 to 8.10]{grothendieck} and \cite[Chapter 2]{forstneric} or \cite[Remark 2.14]{NI2022}.\\

A seminormed line bundle such that the underlying line bundle is ample has an associated structure of semipositive seminormed line bundle. 
Extending Berman's terminology (see \cite{berman2007bergman}) to this formalism, we have the following definition :

\begin{defn}
Let $X$ be a projective scheme over $\Spec \ZZ$. Let $\overline{\mcL} = (\VV_X(\mcL),T)$ be a seminormed line bundle over $X$. Assume furthermore that the underlying line bundle $\mcL$ is ample. The \emph{associated equilibrium seminorm} is the seminorm defined by the holomorphically convex hull $\widehat{T}$ of $T$ in $\VV_X(L)(\CC)$.\\
\end{defn}

\begin{rem} To see why $\widehat{T}$ still defines a seminorm and why this generalises Berman's equilibrium metric, see \cite[Remarks 2.13 and 2.14]{NI2022}.
\end{rem}

%% file: 3-Bornological/Bornology.tex
\subsection{Seminorms on the space of sections of a seminormed vector bundle\\[2mm]}

Let $X$ be a projective scheme over $\Spec \ZZ$ and $\overline{\mcF}=(\mcF, |\cdot|)$ be a seminormed vector bundle over $X$.\\
Then, the space of sections $H^0(X_\CC,\mcF_{\CC})$ can be endowed with different seminorms.\\[2mm]
For example, if $X$ is a projective scheme over $\Spec \ZZ$ such that $X_{\CC}$ is reduced, we can define a seminorm on $H^0(X,\mcF)\otimes_\ZZ \CC$, by setting $$||s||_{\infty} =  \sup_{x\in X(\CC)} |s(x)|_x$$ for $s\in H^0(X,\mcF)\otimes_\ZZ \CC$.\\
In another hand, if, furthermore, we assume $X_{\CC}$ to be smooth, then a Riemannian continuous metric defines a measure $d\mu$ on the complex manifold $X_\CC$, and we can define a seminorm $$||s||_{L^2}^2 = \int_{X_\CC}|s(x)|_x^2 d\mu$$ for $s\in H^0(X,\mcF)\otimes_\ZZ \CC$.\\

It is to be noted that $||\cdot||_\infty$ and $||\cdot||_{L^2}$ have different interesting properties with respect to the Hilbert-Samuel theorem. $||\cdot||_{L^2}$ is a Hilbertian norms which is needed for using $\widehat{\chi}$ which has an useful additivity property with respect to exact sequences, see lemma \ref{additivityclassical}. One the other hand, $||\cdot||_\infty$ is compatible with the multiplication $H^0(X,\mcL^{\otimes n_1})\times H^0(X,\mcL^{\otimes n_2}) \to H^0(X,\mcL^{\otimes n_1 + n_2})$.\\

\subsection{Bornological structure\\[2mm]}
The formalism developed in \cite{Bost-Charles} defines a bornological structure on the algebra $H^0(\VV_X(\mcL), \mcO_\VV) = \bigoplus_n H^0(X,\mcL^{\otimes n})$ associated with $\overline{\mcL}= (\VV_X(\mcL), T)$, denoted $H^0(\VV_X(\mcL), T, \mcO_\VV)$ or $H^0(\VV_X(\overline{\mcL}), \mcO_\VV)$, and unifying the different seminorms mentioned above.\\

Given a complex vector space $V$ and a decreasing family of seminorms $(||\cdot||_i)_{i\in I}$ on $V$, we may define a bornology $\mcB$ on $V$ by defining bounded sets to be the subsets $B$ of $V$ such that there exists $i\in I$ and $R>0$ with
$$\forall v \in B ,||v||_i \leq R$$
In other words, elements of $\mcB$ are the subsets of $V$ that are bounded with respect to one of the
seminorms $||\cdot||_i$, $i \in I$.\\

\begin{defn} Let $V$ be a complex vector space. A \emph{bornology of Hilbertian type on $V$} is a bornology $\mcB$ induced by a decreasing family of Hilbertian seminorms $\mcN = (||\cdot||_i)_{i\in \NN}$ such that the kernel of the $||\cdot||_i$ is eventually constant.
\end{defn}

\begin{rem}
Two families of seminorms $(||\cdot||_i)_{i\in \NN}$ and $(||\cdot||'_j)_{j\in \NN}$ defines the same bornology if:
\begin{itemize}
    \item for $i\in \NN$, there is $j\in \NN$ and $C>0$ such that $||\cdot||'_j \leq C||\cdot||_i$,
    \item for $j\in \NN$, there is $i\in \NN$ and $C>0$ such that $||\cdot||_i \leq C||\cdot||'_j$.\\
\end{itemize}
\end{rem}

\begin{ex}
Let $Y$ be a complex analytic space. Let $K$ be a compact in $Y$. Let $\mcF$ be coherent sheaf on $Y$. Then, $H^0(K,\mcF)$ is defined to be the topological vector space that is the locally convex direct limit of the spaces $H^0(U,\mcF)$, where $U$ runs through the open neighborhoods of $K$ in $Y$, endowed with their natural Fréchet space topology (see \cite[Chapter 8]{GunningRossi}).\\
By pulling back the bornology through the restriction map $H^0(Y,\mcF) \to H^0(K,\mcF)$, the complex vector space $H^0(Y,\mcF)$ is then endowed with a bornology of Hilbertian type, see \cite{Bost-Charles}.\\
\end{ex} 

\begin{defn}
Let $Y$ be a complex analytic space. Let $K$ be a compact in $Y$. Let $\mcF$ be coherent sheaf on $Y$. The complex space $H^0(Y,\mcF)$ endowed with the bornology described above is denoted $H^0(Y,K,\mcF)$.
\end{defn}

\begin{defn}
An \emph{arithmetic Hilbertian $\mcO_{\Spec \ZZ}$-module} $\overline{M}$ is a countably generated $\ZZ$-module $M$, such that $M_{\CC}=M\otimes_\ZZ \CC$ is endowed with a bornology of Hilbertian type, invariant by complex conjugation.\\
Furthermore we say that $\overline{M}$ is a \emph{graded Hilbertian $\mcO_{\Spec \ZZ}$-module}, if $M$ is graded.\\
\end{defn}

\begin{rem}
The invariant by conjugation property ensures that the bornology can be defined by using a decreasing family of Hilbertian seminorms invariant by conjugation, see \cite{Bost-Charles}.
\end{rem}

Let $\mcF$ be coherent sheaf on a separated scheme $Y$ of finite type over $\Spec \ZZ$. Let $K$ be an invariant by complex conjugation compact in the complex analytic space $Y(\CC)$.\\
Then $H^0(Y, \mcF)$ is endowed with a structure of arithmetic Hilbertian $\mcO_{\Spec \ZZ}$-module, where the bornology on $H^0(Y, \mcF)_\CC = H^0(Y(\CC), \mcF_{|\CC})$ is given by the compact $K$ as above.\\

\begin{defn}
Let $Y$ be a separated scheme of finite type over $\Spec \ZZ$. Let $K$ be an invariant by complex conjugation compact in the complex analytic space $Y(\CC)$. Let $\mcF$ be a coherent sheaf on $Y$.\\
The arithmetic Hilbertian $\mcO_{\Spec \ZZ}$-module structure on $H^0(Y, \mcF)$ described above is denoted $H^0(Y,K,\mcF)$.\\
In the case of a seminormed line bundle $\overline{\mcL}= (\VV_X(\mcL),T)$, the graded arithmetic Hilbertian $\mcO_{\Spec \ZZ}$-module $H^0(\VV_X(\mcL),T,\mcF)$ is also denoted $H^0(\VV_X(\overline{\mcL}),\mcF)$.\\
\end{defn}

\begin{rem} The invariance by conjugation of the bornology associated to $H^0(Y,K,\mcF)$ arises from the invariance by conjugation of the compact $K$, see \cite{Bost-Charles}.
\end{rem}

\subsection{Reduced and smooth cases\\[2mm]}

In the case of a reduced complex analytic space $Y$, $K$ a compact in $Y$ and $\mcF$ a coherent sheaf on $Y$. We have a more explicit description of the bornology associated to $H^0(Y, K,\mcF)$, which makes the relationship with the $L^2$-norms and $L^\infty$-norms more explicit.\\

\begin{prop}\label{reducedcasenomrlp}
 Let $Y$ be a reduced complex analytic space, let $K$ be a compact subset of $Y$, and let $\mcF$ be a Hermitian vector bundle over $Y$.\\
 Then the bornology on $H^0(K, \mcF)$ is induced by the limit of $L^\infty$-norms on relatively compact neighborhoods of $K$ in $Y$.\\
Furthermore, if we assume that $Y$ is smooth, then by choosing a Riemannian continuous metric on $Y$, we can define $L^p$-norms, for $1\leq p<\infty$. \\
Then the bornology on $H^0(K, \mcF)$ is induced by the limit of $L^p$-norms on relatively compact neighborhoods of $K$ in $Y$, for $1\leq p\leq \infty$.
\end{prop}

\begin{proof}
See \cite{Bost-Charles}.
\end{proof}

In the case of the total space of a line bundle, we have the following description
\begin{prop}\label{reducedcaseBornology}
Let $X$ be a smooth projective scheme over $\Spec \CC$, $\overline{\mcL}$ be a Hermitian line bundle over $X$, $\overline{\mcF}$ be a Hermitian vector bundle over $X$. Choose a Riemannian continuous metric on $X$.\\
Take $1 \leq p \leq \infty$.
Let $p_X: \VV_X(\mcL)\to X$ be the total space of $\mcL^{\vee}$.\\
Assume $(||\cdot||_j)_{j\in \NN}$ is a decreasing family of norms defining the bornology on $H^0(\VV_X(\overline{\mcL}),p_X^*\mcF)$. \\
Then, fix $j\in \NN$, there is $\epsilon>0$, $C>0$ such that, for all $n\in \NN$,
$$Ce^{n\epsilon} ||\cdot||_{L^p}\leq ||\cdot||_j \text{ , on }  H^0(X, \mcL^{\otimes n} \otimes \mcF)_\CC $$
Fix $\epsilon>0$, there is $j \in \NN$, $C>0$ such that, for all $n\in \NN$,
$$C||\cdot||_j \leq e^{n\epsilon} ||\cdot||_{L^p} \text{ , on }  H^0(X, \mcL^{\otimes n} \otimes \mcF)_\CC $$
\end{prop}

\begin{proof}
See \cite{Bost-Charles}.
\end{proof}

\subsection{Dependence of the bornology on the seminorm\\[2mm]}

The following proposition states that, on an ample line bundle, a seminorm and its associated equilibrium seminorm induce the same bornology on the complex vector space $H^0(\VV_X(\mcL)(\CC),\mcO_\VV)$.\\

\begin{prop}\label{passageAlenveloppeholo}
Let $X$ be a reduced projective space over $\Spec \ZZ$. Let $\overline{\mcL}= (\VV_X(\mcL),T)$ be a seminormed line bundle over $X$. Assume furthermore that $\mcL$ is ample. 
Then the restriction morphism $H^0(\VV_X(\mcL),\widehat{T},\mcO_\VV) \to H^0(\VV_X(\mcL),T,\mcO_\VV)$ is an isomorphism.
\end{prop}

\begin{proof}
See \cite[proposition 4.15]{NI2022}.
\end{proof}

This gives a glimpse on why the semipositivity assumption can be dropped in the arithmetic Hilbert-Samuel theorem when the coherent sheaf $\mcF$ is $\mcO_X$.

%% file: 5-Proof/introproof.tex
We first give a proof of the geometric Hilbert-Samuel theorem using the deformation to the projective completion of the cone, classical reductions in the theory of coherent sheaves, and the case of $(\PP^l_\ZZ,\mcO_\PP(1),\mcO_\PP)$. Then, we return to the arithmetic Hilbert-Samuel theorem, prove the case of $\PP^l_\ZZ$ with the line bundle $\mcO_\PP(1)$ endowed with a seminorm invariant under the action of $U(1)^{l+1}$, and finally give a proof along the same lines as the proof of the geometric Hilbert-Samuel theorem.\\
In our proof, we use the arithmetic Hilbert invariants described in \cite{NI2022}. Those are numerical invariants associated to bornological spaces $H^0(\VV_X(\mcL),p_X^*\mcF)$ and they inherit the functorial properties of this bornological structures. We will also use some special case of the deformation to the projective completion of the cone developed in \cite{NI2022}. All the necessary facts are extracted from \cite{NI2022} and given in appendix.

%% file: 5-Proof/GeometricHilbertSamuel.tex
\subsection{The geometric Hilbert-Samuel theorem}\label{geometricproof}
The idea is to use the deformation to the projective completion of the cone to reduce the proof to the case of $(\PP^{l}_{\ZZ}, \mcO_\PP(1),\mcO_\PP)$. This geometric proof is inspired by the classical proof of the Hilbert-Samuel theorem by the "dévissage" process. \\

We first give the definition of the Hilbert function:\\

\begin{defn}
Let $X$ be a projective scheme over $\Spec k$. Let $\mcL$ be a very ample line bundle over $X$ and $\mcF$ be a coherent sheaf over $X$. \\
Then the \emph{Hilbert function} associated to $(X,\mcL,\mcF)$ is $H(n) := H_{X,\mcL,\mcF}(n) = \dim_{k} H^0(X,\mcF\otimes \mcL^{\otimes n})$ for $n \in \NN$.\\
\end{defn}

The invariance by deformation given in corollary \ref{invariancepardeformation}, will play a significant role in our proof of the geometric Hilbert-Samuel theorem :\\

\begin{thm}\label{hilsamgeo}
Let $X$ be a projective scheme over $\Spec k$. Let $\mcL$ be a very ample line bundle over $X$. Let $\mcF$ be a coherent sheaf over $X$. \\
Then the Hilbert function associated to $(X,\mcL,\mcF)$ is a numerical polynomial for $n\gg 0$.
\end{thm}

Before proving theorem \ref{hilsamgeo}, let's describe a few reductions: \\

\begin{lemma}\label{CompatibilitySubspaces}Let $X$ be a projective scheme over $\Spec k$. Let $\mcL$ be a line bundle over $X$.\\
If $\iota : Y \to X$ is a closed immersion and $\mcF$ is a coherent sheaf on $Y$, then there are isomorphisms between  $H^0(X, \mcL^{\otimes n} \otimes \iota_*\mcF)$ and $H^0(Y, \iota^*\mcL^{\otimes n} \otimes \mcF )$.\\
In particular, $(X,\mcL,\iota_* \mcF)$ and $(Y,\iota^*\mcL, \mcF)$ have the same Hilbert functions.
\end{lemma}

\begin{proof}
This is the projection formula and compatibility of pullbacks and tensor products, see \cite[chapter 16]{Vakil}.\\
\end{proof}

\begin{lemma}\label{filtrationdeviss} Let $X$ be a projective scheme over $\Spec k$. Let $\mcL$ be an ample line bundle over $X$, $\mcF$ a coherent sheaf on $X$.\\
Let $0 = \mcF_0 \subset \mcF_1 \subset \dots \subset \mcF_r = \mcF$ be a filtration of $\mcF$ by coherent subsheaves, then for $n\gg 0$, the Hilbert function of $(X,\mcL,\mcF)$ is the sum of the Hilbert function of $(X,\mcL,\mcF_i/\mcF_{i-1})$ for $1 \leq i \leq r$.
\end{lemma}

\begin{proof}
For $n\gg 0$, we have a short exact sequence $0 \to H^0(X,\mcL^{\otimes n}\otimes \mcF_{i-1})\to H^0(X,\mcL^{\otimes n}\otimes \mcF_{i}) \to H^0(X,\mcL^{\otimes n}\otimes \mcF_i/\mcF_{i-1}) \to  0$. By additivity of $\dim_k$, this concludes.\\
\end{proof}

The additivity of the Hilbert function for $n\gg 0$ over exact sequences of coherent sheaves allows useful reductions : \\

\begin{lemma}\label{ensembledesreductions} Let $X$ be a projective scheme over $\Spec k$. Let $\mcL$ be a very ample line bundle over $X$, $\mcF$ a coherent sheaf on $X$. In high degree, 
\begin{enumerate}
    \item if the ideal sheaf $\mcI$ define the underlying reduced space $X_{red}$, and $\mcI^r= 0$, then the Hilbert function of $(X,\mcL,\mcF)$ is the sum of the Hilbert functions of $(X_{red}, \mcL_{|X_{red}}, \mcI^j\mcF/\mcI^{j+1}\mcF)$ for $0\leq j \leq r-1$.
    \item the Hilbert function of $(X,\mcL,\mcF)$ is an alterned sum of Hilbert functions of $(X,\mcL,\mcO_X)$ shifted by some integers.
    \item the Hilbert function of $(X,\mcL,\mcF)$ is a sum of Hilbert functions of triples of the form $(V,\mcL_{|V}, \mcI)$ where $V$ is an integral closed subscheme of $X$, and $\mcI$ is an ideal sheaf on $V$.
\end{enumerate}
\end{lemma}

\begin{proof}
For point 1, as there is a filtration $0 = \mcI^r\mcF \subset \mcI^{r-1}\mcF \subset \dots \subset \mcI\mcF \subset \mcF$ and as $\mcI^j\mcF/\mcI^{j+1}\mcF$ are supported on $X_{red}$, this is a consequence of lemmas \ref{filtrationdeviss} and \ref{CompatibilitySubspaces}.  \\ 
For point 2, this is a consequence of the Hilbert Syzygy theorem (see \cite{eisenbud}) that there is a finite resolution of $\mcF$ by coherent sheaves of the forms $\bigoplus_i \mcL^{\otimes l_i}$.\\
For point 3, there is a filtration $0= \mcF_0 \subset \mcF_1 \subset \dots \mcF_r = \mcF$ by coherent subsheaves such that for $1\leq j \leq r$, there is an integral subscheme $\iota_j : V_j \to X$ and an ideal sheaf $\mcI_j \subset \mcO_{V_j}$ such that $\mcF_j/\mcF_{j+1} \simeq {\iota_j}_* \mcI_j$, see    
\cite[\href{https://stacks.math.columbia.edu/tag/01YF}{Lemma 01YF}]{stacks-project}. Thus, by lemma \ref{filtrationdeviss}, this concludes.\\
\end{proof}

\begin{proof}[Proof of theorem \ref{hilsamgeo}]
Firstly, by point $2$ of \ref{ensembledesreductions}, we can assume that $\mcF = \mcO_X$.\\
By applying the deformation to the projective completion of the cone repetitively by keeping track of closed immersions to $\PP^N_k$, using corollary \ref{CombinaisonActionGm} and lemma \ref{invariancepardeformation}, we can assume that there is a closed immersion induced by sections of $\mcL$, $X\to \PP^N_k = \Proj k[x_0, \dots, x_N]$ such that for all $0\leq i \leq N$, $X$ is stable by the action of $\GG_m$ by multiplication on $x_i$. Furthermore, by basic considerations on the Rees algebras, the dimension is non-increasing by this deformations, see \cite[ theorem 15.7]{matsumura_1987}.\\
By noting then that the surjection $k[x_0, \dots, x_N] \to \bigoplus H^0(X,\mcL^{\otimes n})$ is graded in each variable $x_i$, its kernel is an ideal generated by monomes $\underline{x}^{\underline{i}}$, the irreducible components of $X$ are thickening of $\PP^l_k$.\\
By point $3$ of \ref{ensembledesreductions}, we are reduced to prove the theorem for the integral subschemes $V_i$ of $X$ and $\mcF$ is an ideal sheaf of $\mcO_{V_i}$. By point $2$ of \ref{ensembledesreductions} again, we can assume that $\mcF = \mcO_{V_i}$.\\
Then, two cases arise, either $\dim V_i < \dim X$ and we conclude by induction, either $\dim V_i = \dim X$ and $V_i$ is isomorphic to $\PP^l_k$ with $l = \dim X$.\\
Finally, by using point $2$ of \ref{ensembledesreductions}, we are reduced to the case $(\PP^l_k, \mcO_\PP(1), \mcO_\PP)$. And, in this case the theorem is true as $H_{(\PP^l_k, \mcO_{\PP}(1), \mcO_\PP)}(n) = \binom{n+l}{l}$, for $n\geq 0$.
\end{proof}

This proof relies entirely on geometric constructions and the knowledge of the case of $(\PP^l_k, \mcO_{\PP}(1), \mcO_\PP)$. In the arithmetic case, we will prove that the same reductions are possible, and this will lead in exactly the same manner to a proof of the arithmetic Hilbert-Samuel theorem. \\

%% file: ProofArithmeticHilbertSamuel/Keytechnicallemma.tex
\subsection{The case of $\PP^l_\ZZ$.}\label{CaseofPnO1}
The standard case of the arithmetic Hilbert Samuel theorem is the case of $\PP^l_\ZZ = \Proj \ZZ[x_0, \dots x_l]$, with the line bundle $\mcO_\PP(1)$ endowed with a seminorm stable by the action of $U(1)^{l+1}$ by multiplication on $x_0, \dots , x_l$.\\
In this case, the arithmetic Hilbert-Samuel theorem is a consequence of Fekete's superadditivity lemma. \\
In this paragraph, we will furthermore prove theorem \ref{CasSupplementairedePn}, by using an explicit version of the deformation to the projective completion of the cone argument, which allows us to bypass the uniformly definite hypothesis in this case. In this case, the main idea is that the deformation to the projective completion of the cone do not modify $\PP^l_\ZZ$ but only act on the analytic tube.\\

\subsubsection{A superadditivity lemma.} Firstly, in order to use Fekete's lemma effectively, we will need the following proposition on the stability of the superadditivity property by taking means.\\

\begin{prop}\label{superadditive}
Let $r, l\in \NN$.\\
Let $(\alpha_{\underline{i},\underline{j}})_{(\underline{i},\underline{j})\in \NN^{r+1}\times \NN^l}$ be a sequence of real numbers which is superadditive for the monoid law of $\NN^{r+1+l}$.\\ Define $\beta_{n,\underline{j}} = \displaystyle\frac{1}{P_r(n)} \sum_{|\underline{i}|=n} \alpha_{\underline{i},\underline{j}}$ with $P_r(n) = 	\displaystyle\binom{n+r}{r} = \# \{\underline{i}\in \NN^r \;|\; |\underline{i}|= i_0 + \dots + i_{r} = n\}$. 
\\Then
$(\beta_{n,\underline{j}})_{(n,\underline{j})\in \NN^{1+l}}$ is superadditive for the monoid law of $\NN^{1+l}$, that is
$$ \beta_{n,\underline{j}} + \beta_{m,\underline{j'}} \leq \beta_{n+m,\underline{j}+\underline{j'}} \text{ for all } (n,\underline{j}) \text{ and } (m,\underline{j'}) \text{ in } \NN^{1+l}$$
\end{prop}

\begin{proof}
We will prove the proposition by working by induction on $r$.\\
Notice that the claim is trivial for $r=0$ and $l\in \NN$.\\
Now, let $r > 0$, $l\in \NN$ and $(\alpha_{\underline{i},\underline{j}})_{(\underline{i},\underline{j})\in \NN^{r+1+l}}$ be a superadditive sequence.\\

 We have
$\beta_{n,\underline{j}} =  \displaystyle\frac{1}{P_r(n)} \sum_{|\underline{i}|=n} \alpha_{\underline{i},\underline{j}} = \frac{1}{P_r(n)} \sum\limits_{i_{r}=0}^n \sum_{|\underline{i'}|=n-i_{r}} \alpha_{\underline{i'},i_{r},\underline{j}} = \sum\limits_{i_{r}=0}^n \lambda_{n,i_{r}}\xi_{n-i_{r},i_{r},\underline{j}} $\\
where $\lambda_{n,i} = \displaystyle\frac{P_{r-1}(n-i)}{P_r(n)}$ and where $\xi_{i_0,i_1,\underline{j}} = \displaystyle\frac{1}{P_{r-1}(i_0)}\sum_{|\underline{i'}|=i_0} \alpha_{\underline{i'},i_1,\underline{j}}$ is a superadditive sequence on $\NN^{1+(1+l)}$ .\\

Let $n,m\in \NN$ and $\underline{j},\underline{j}' \in \NN^l$. Let's show that $$ \beta_{n,\underline{j}} + \beta_{m,\underline{j'}} \leq \beta_{n+m,\underline{j}+\underline{j'}} $$
The existence of a matrix $M \in M_{n+1,m+1}(\RR_{\geq 0})$ such that 
$$\displaystyle\sum_{i=0}^n m_{i,i'} = \lambda_{m,i'} \; ,\; \displaystyle\sum_{i'=0}^m m_{i,i'} = \lambda_{n,i} \text{ and } \displaystyle\sum_{i'+i=k} m_{i,i'} = \lambda_{n+m,k}$$
is proven in the lemma below.\\ 
Then we have,
\begin{align*}
\beta_{n,\underline{j}} + \beta_{m,\underline{j'}} &= \sum\limits_{i=0}^n \lambda_{n,i} \xi_{n-i,i,\underline{j}} + \sum\limits_{i'=0}^m \lambda_{m,i'} \xi_{m-i',i',\underline{j'}} \\
&= \sum_{i,i'} m_{i,i'}(\xi_{n-i,i,\underline{j}} + \xi_{m-i',i',\underline{j'}})\\
&\leq \sum_{i,i'} m_{i,i'}\xi_{n+m-i-i',i+i',\underline{j}+\underline{j'}}\\
&\leq \sum_{k=0}^{n+m} \lambda_{n+m,k}\xi_{n+m-k,k,\underline{j}+\underline{j'}}\\
&\leq \beta_{n+m,\underline{j}+\underline{j'}}
\end{align*}
This concludes.
\end{proof}

\begin{lemma}
There exists $M\in M_{n+1,m+1}(\RR_{\geq 0})$ such that $$\displaystyle\sum_{i=0}^n m_{i,i'} = \frac{P_{r-1}(i')}{P_r(m)} \; ,\; \displaystyle\sum_{i'=0}^m m_{i,i'} = \frac{P_{r-1}(i)}{P_r(n)} \text{ and } \displaystyle\sum_{i'+i=k} m_{i,i'} = \frac{P_{r-1}(k)}{P_r(n+m)}$$
with $P_r(n) = 	\displaystyle\binom{n+r}{r}$ and $r > 0$ .
\end{lemma}

\begin{proof} We proceed by induction.
If $n=0$, the result is clear. Indeed, $m_{0,i'} = \displaystyle\frac{P_{r-1}(i')}{P_r(m)}$ meets the required conditions  as $\displaystyle\sum_{i'}\frac{P_{r-1}(i')}{P_r(m)} = 1$. The same goes if $m=0$.\\

Let $n,m \geq 1$.\\
Let $M_{n,m+1}$ be a matrix satisfying the above conditions and considered as an element of $M_{n+1,m+1}(\RR)$ by completing by a line of $0$. \\
Let $M_{n+1,m}$ be a matrix satisfying the above conditions and considered as an element of $M_{n+1,m+1}(\RR)$ by completing by a column of $0$. \\

We claim that the matrix  $$ M = \frac{P_{r-1}(n+m)}{P_{r}(n+m)}\left( \frac{P_r(n-1)}{P_{r-1}(n)} M_{n-1,m} +  \frac{P_r(m-1)}{P_{r-1}(m)} M_{n,m-1} + E_{n,m}\right) $$
verifies the conditions for $(n,m)$.\\
By symmetry, we only need to verify :
\begin{enumerate}
    \item[-] for the first $n$ lines, that : $\frac{P_{r-1}(n+m)}{P_{r}(n+m)} \left( \frac{P_r(n-1)}{P_{r-1}(n)}\frac{P_{r-1}(i)}{P_r(n-1)} + \frac{P_r(m-1)}{P_{r-1}(m)}\frac{P_{r-1}(i)}{P_r(n)}\right) = \frac{P_{r-1}(i)}{P_r(n)}$.
    \item[-] for the last line, that : $\frac{P_{r-1}(n+m)}{P_{r}(n+m)} \left( \frac{P_r(m-1)}{P_{r-1}(m)}\frac{P_{r-1}(n)}{P_r(n)} + 1 \right) = \frac{P_{r-1}(n)}{P_r(n)}$.
    \item[-] for the first $m+n$ anti-diagonals, that : $\frac{P_{r-1}(n+m)}{P_{r}(n+m)}\left( \frac{P_r(n-1)}{P_{r-1}(n)} \frac{P_{r-1}(k)}{P_r(n+m-1)} +  \frac{P_r(m-1)}{P_{r-1}(m)} \frac{P_{r-1}(k)}{P_r(n+m-1)}\right) = \frac{P_{r-1}(k)}{P_r(n+m)}$.
    \item[-] for the last anti-diagonal : it is trivial.
\end{enumerate}
But, this is clear from the formulas :
 $\frac{P_r(n)}{P_{r-1}(n) } + \frac{P_r(m-1)}{P_{r-1}(m)}  = \frac{P_r(n+m)}{P_{r-1}(n+m) }$ and $\frac{P_r(n-1)}{P_{r-1}(n) } + \frac{P_r(m-1)}{P_{r-1}(m)}  = \frac{P_r(n+m-1)}{P_{r-1}(n+m) }$ .
\end{proof}

As announced, we have the following consequence:

\begin{cor}\label{covergenceparsuadditivite}
Let $(\alpha_{\underline{i}})_{\underline{i}\in \NN^{r+1}}$ be a superadditive sequence. Then the sequence  $$\frac{1}{n^{r+1}}\sum_{|\underline{i}| = n } \alpha_{\underline{i}}$$
converges in $\RR \cup \{+\infty\}$, when $n\to \infty$.
\end{cor}

\begin{proof}
By proposition \ref{superadditive}, $\left( \frac{1}{P_r(n)}\sum_{|\underline{i}| = n } \alpha_{\underline{i}}\right)$ is a superadditive sequence. Then by Fekete's superadditivity lemma, $\left( \frac{1}{nP_r(n)}\sum_{|\underline{i}| = n } \alpha_{\underline{i}}\right)$  converges in $\RR \cup \{+\infty\}$, when $n\to \infty$.
\end{proof}

%% file: ProofArithmeticHilbertSamuel/TheCaseOfPn.tex
\subsubsection{Proof of the case of $\PP^l_\ZZ$ }

\begin{prop}\label{CasDePn} Assume that the line bundle $\mcO_\PP(1)$ over $\PP^l_\ZZ = \Proj \ZZ[x_0, \dots, x_l]$ is endowed with a structure of seminormed line bundle such that the associated analytic tube is stable by the action of $U(1)^{l+1}$ by multiplication on $x_0, \dots, x_l$. \\
Then, we have $$\underline{c}_{ l+1}(\PP^l_\ZZ, \overline{\mcO_\PP(1)}, \mcO_\PP) = \overline{c}_{ l+1}(\PP^l_\ZZ, \overline{\mcO_\PP(1)},\mcO_\PP) \in \RR\cup\{+\infty\}$$

\end{prop}

One of the main point of the proof is the compatibility of the bornological structure and the algebra structure as stressed in \cite{houzel}.
\begin{proof}
By \cite{Bost-Charles}, the bornology on $H^0(\VV_{\PP^l_\ZZ}(\overline{\mcO_\PP(1)}), \mcO_\VV)$ is defined by a family of Hilbertian seminorms $(||\cdot||_i)_{i\in \NN}$ invariant by complex conjugation, orthogonal for the actions of $\GG_m$. That is, the monomes $(\underline{x}^{\underline{m}})_{\underline{m} \in \NN^{l+1}}$ form orthogonal basis for the seminorms $||\cdot||_i$. Hence $$\widehat{\chi}(H^0(\PP^l_\ZZ,\mcO_\PP(n)),||\cdot||_i) = \sum\limits_{|\underline{m}|=n} -\log ||\underline{x}^{\underline{m}}||_i$$
Now, choose a decreasing basis of relatively compact neighborhoods $(U_j)_{j\in \NN}$ of $T \subset \VV_{\PP^l_\ZZ}(\mcO_\PP(1))$, by \cite{Bost-Charles}, the decreasing family of norms $(||\cdot||_{\infty, U_j})_{j\in \NN}$ defines the bornology on $H^0(\VV_{\PP^l_\ZZ}(\overline{\mcO_\PP(1)}), \mcO_\VV)$.
Furthermore, those norms are algebra norms. In particular, they verify that $||x^{\underline{m}+ \underline{m'}}||_{\infty, U_j} \leq ||x^{\underline{m}}||_{\infty, U_j}|| x^{\underline{m'}}||_{\infty, U_j}$.\\
By corollary \ref{covergenceparsuadditivite}, this ensures that $\frac{1}{n^{l+1}}\sum_{|\underline{m}|=n}-\log||\underline{x}^{\underline{m}}||_{\infty, U_j}$ converges.\\
By using the comparison between two families of norms defining the same bornology, we get that 
$$\underline{c}_{l+1}(\PP^l_{\ZZ}, \overline{\mcO_\PP(1)}, \mcO_\PP) =\overline{c}_{l+1}(\PP^l_{\ZZ}, \overline{\mcO_\PP(1)}, \mcO_\PP) = \lim\limits_{j\to \infty}\lim\limits_{n\to \infty} \frac{(l+1)!}{n^{l+1}}\sum_{|\underline{m}|=n}-log||\underline{x}^{\underline{m}}||_{\infty, U_j}$$
\end{proof}

Using the deformation to the projective completion of the cone, we immediately have the following corollary: 
\begin{cor}
 Assume that the line bundle $\mcO_\PP(1)$ over $\PP^l_\ZZ$ is endowed with a structure of uniformly definite seminormed line bundle.\\
Then, we have $$\underline{c}_{ l+1}(\PP^l_\ZZ, \overline{\mcO_\PP(1)}, \mcO_\PP) = \overline{c}_{ l+1}(\PP^l_\ZZ, \overline{\mcO_\PP(1)},\mcO_\PP) \in \RR\cup \{+\infty\}$$
\end{cor}

\begin{proof}
Let $\PP^l_\ZZ = \Proj \ZZ[x_0,\dots,x_l]$. Then the fiber above $\infty$ of the deformation to the projective completion of the cone associated to $(\PP^l_\ZZ, \mathrm{div}(x_0),\overline{\mcO_\PP(1)})$ is $(\PP^l_\ZZ = \Proj \ZZ[x'_0,x_1,\dots,x_l],\mcO_\PP(1))$. Hence, it deforms $(\PP^l_\ZZ,\mcO_\PP(1))$ into $(\PP^l_\ZZ,\mcO_\PP(1))$ but the analytic tube associated to the deformation is stable by the action of the natural transverse action of $U(1)$. At the fiber above $\infty$, the natural transverse action is the multiplication on $x'_0$, by proposition \ref{CombinaisonActionGm}. Hence, by the conservation by deformation of the arithmetic Hilbert invariants (proposition \ref{invarianceArithmeticHilbert}), we can assume that the analytic tube of $\overline{\mcO_\PP(1)}$ is stable by the action of $U(1)$ by multiplication on the first variable.
By doing this operation repetively, using corollary \ref{CombinaisonActionGmU1}, we can assume that the analytic tube associated to $\overline{\mcO_\PP(1)}$ is stable by the action of $U(1)^{l+1}$ by multiplication on $x_0, \dots, x_l$.\\
Then, proposition \ref{CasDePn} concludes.
\end{proof}

It is possible to give an extension of this theorem in the non uniformly definite case.
 The deformation to the projective completion of the cone deforms $\PP^l_\ZZ$ into $\PP^l_\ZZ$ and adds an action of $U(1)$ on the analytic tube of $\overline{\mcO_\PP(1)}$. However, the conservation of the arithmetic Hilbert invariants given in proposition \ref{invarianceArithmeticHilbert} cannot be used when we are not considering uniformly definite normed line bundle. Hence to prove theorem \ref{CasSupplementairedePn}, we give an explicit argument which can be understood as an explicit deformation argument.\\

\begin{cor}
 Assume that the line bundle $\mcO_\PP(1)$ over $\PP^l_\ZZ$ is endowed with a structure of seminormed line bundle.\\
Then, we have $$\underline{c}_{ l+1}(\PP^l_\ZZ, \overline{\mcO_\PP(1)}, \mcO_\PP) = \overline{c}_{ l+1}(\PP^l_\ZZ, \overline{\mcO_\PP(1)},\mcO_\PP) \in \RR\cup \{+\infty\}$$
\end{cor}

\begin{proof}
By \cite{Bost-Charles}, the bornology on $H^0(\VV_{\PP^l_\ZZ}(\overline{\mcO_\PP(1)}), \mcO_\VV)$ is defined by a family of Hilbertian seminorms $(||\cdot||_i)_{i\in \NN}$.\\
Choose a decreasing basis of relatively compact neighborhoods $(U_j)_{j\in \NN}$ of $T \subset \VV_{\PP^l_\ZZ}(\mcO_\PP(1))$, by \cite{Bost-Charles}, the decreasing family of norms $(||\cdot||_j')_{j\in \NN} =(||\cdot||_{\infty, U_j})_{j \in \NN}$ defines the bornology on $H^0(\VV_{\PP^l_\ZZ}(\overline{\mcO_\PP(1)}), \mcO_\VV)$.\\
Now, take $I =x_{m} H^0(\VV_{\PP^l_\ZZ}(\mcO_\PP(1)), \mcO_\VV) = x_{m} \ZZ[x_0,\dots, x_l] $, and consider the bornology on $\bigoplus I^k/I^{k+1}$ given by the family of subquotient seminorms $(||\cdot||_{i,sq})_{i\in \NN}$. This bornology is also given by the family $(||\cdot||'_{j,sq})_{j\in \NN}$.\\
Furthermore, $\bigoplus I^k/I^{k+1}$ is isomorphic to the ring of polynomial in $l+1$ variables.\\
By additivity of $\widehat{\chi}$ given in lemma \ref{additivityclassical}, and by doing the previous operation repetitively, we can assume that we have the ring $\ZZ[x_0, \dots, x_l]$ endowed with a bornology given by two families of seminorms, $(||\cdot||_i)_{i\in \NN}$ a family of Hilbertian seminorms orthogonals on the monomes $\underline{x}^{\underline{k}}$ and $(||\cdot||'_j)_{j\in \NN}$ a family of seminorms
which verifies that $||\underline{x}^{\underline{k+k'}}||_j' \leq ||\underline{x}^{\underline{k}}||_j' ||\underline{x}^{\underline{k'}}||_j'$.\\
Then, we can conclude as in the proof of the proposition \ref{CasDePn}.
\end{proof}

\subsubsection{Finiteness of the arithmetic Hilbert invariants.} The uniformly definite condition ensures an upper bound on the arithmetic Hilbert invariants:

\begin{prop}\label{minorationlambda} Choose a continuous Riemannian metric on $\PP^l_\CC$. Assume that $\mcO_\PP(1)$ is endowed with a uniformly definite metric. The first mininum $\lambda_1(H^0(\PP^l_\ZZ, \mcO_\PP(n)), ||\cdot||_{L^2})$ verifies that there is $c>0$ such that for all $n\in 0$, we have $$\lambda_1(H^0(\PP^l_\ZZ, \mcO_\PP(n)), ||\cdot||_{L^2})\geq c^n$$
\end{prop}

\begin{proof}
As the metric is uniformly definite, it is bigger than a multiple of the Fubini-Study metric. As scaling the metric by $\lambda>0$, scale the norm $||\cdot||_{L^2}$ by $\lambda^n$ on $H^0(\PP^l_\ZZ, \mcO_\PP(n))$, we can assume that the metric is the Fubini-Study metric. \\
Let's first prove the statement for the sup norm $||\cdot||_{\infty}$.
Let $P = \sum_{\underline{i}}a_{\underline{i}}\underline{x}^{\underline{i}} \in H^0(\PP^l_\ZZ, \mcO_\PP(n)) = \ZZ[x_0,\dots, x_l]_n$. 
By integrating over $K = \{x_0 = e^{i\theta_0}, \dots , x_{l} = e^{i\theta_l}\}$, we get that $\sup_{K} |P| \geq \left(\sum_{\underline{i}}|a_{\underline{i}}|^2\right)^{1/2}  \geq 1$.
Hence, 
$$||P||_{\infty} = \sup_{[x_0:\dots : x_l]}\frac{|P(x_0, \dots, x_l)|}{(|x_0|^2+ \dots +|x_l|^2)^{n/2}}\geq \frac{1}{(l+1)^{n/2}}$$
Then, there is $\epsilon> 0$ and $C>0$ such that $||\cdot||_{L^2}\geq Ce^{-\epsilon n}||\cdot ||_{\infty}$ on $H^0(\PP^l_\ZZ, \mcO_\PP(n))$, by proposition \ref{reducedcaseBornology}.\\
Then, this concludes.
\end{proof}

\begin{cor}\label{finitenesslimit}
 Assume that $\mcO_\PP(1)$ is endowed with a uniformly definite metric. Then $\overline{c}_{l+1}(\PP^l_\ZZ, \overline{\mcO_\PP(1)},\mcO_\PP) < + \infty$.\\
 Furthermore, if $r>l+1$, $\overline{c}_{r}(\PP^l_\ZZ, \overline{\mcO_\PP(1)},\mcO_\PP)=0$.
\end{cor}

\begin{proof}
As in \ref{minorationlambda}, because we aim for an upper bound to $\overline{c}_{l+1}(\PP^l_\ZZ, \overline{\mcO_\PP(1)},\mcO_\PP) < + \infty$ and by the monotony property given in proposition 
\ref{uniquenessinvariant}, we can assume that the metric is the Fubini-Study metric.\\
Firstly, choose a continuous Riemannian metric on $\PP^l_\CC$. Then, as the metric is continuous, by proposition \ref{uniquenessinvariant}, we have  $$\overline{c}_{l+1}(\PP^l_\ZZ, \overline{\mcO_\PP(1)}) = \limsup_n \frac{(l+1)!}{n^{l+1}}\widehat{\chi}(H^0(\PP^l_\ZZ, \mcO_\PP(n)),||\cdot||_{L^2}) $$
Let $d_n = \mathrm{rk} H^0(\PP^l_\ZZ, \mcO_\PP(n)) = \binom{l+n}{l}$, $\widehat{\chi}_n = \widehat{\chi}(H^0(\PP^l_\ZZ, \mcO_\PP(n)),||\cdot||_{L^2}) $ and $\lambda_{1,n} = \lambda_1(H^0(\PP^l_\ZZ, \mcO_\PP(n)), ||\cdot||_{L^2})$.\\
By Minkowski's first theorem, we have $\lambda_{1,n} \leq 2\exp(-\widehat{\chi}_n/d_n) v_{d_n}^{-1/d_n} $, where $v_N$ denote the volume of the unit ball in $\RR^N$.\\
Thus, $\frac{1}{n d_n}\widehat{\chi}_n \leq \frac{1}{n}\log 2  - \frac{1}{n}\log \lambda_{1,n} -\frac{1}{n d_n}\log v_{d_n} $.\\
As $-\frac{1}{N}\log v_N \sim_N \frac{1}{2}\log N$ and by proposition \ref{minorationlambda}, this concludes.
\end{proof}

%% file: ProofArithmeticHilbertSamuel/TheProof.tex
\subsection{Conclusion : proof of the arithmetic Hilbert-Samuel theorem. }\label{Proooooooof}
The idea is to use exactly the same reductions as in the demonstration of the geometric Hilbert-Samuel theorem in paragraph \ref{geometricproof} but taking into account the analytic tube and its deformations to prove the arithmetic Hilbert-Samuel theorem.\\
The process of the successive deformations and reductions will then lead to the standard case of $\PP^l_\ZZ = \Proj \ZZ[X_0, \dots , X_l]$ with the line bundle $\mcO_\PP(1)$ endowed with a seminorm stable by the action of $U(1)^{l+1}$ by multiplication on $X_0,\dots , X_l$.\\
We prove: 

\begin{thm}\label{thprincipalbis} Let $X$ be a projective scheme over $\Spec \ZZ$ of dimension $d$. Let $\overline{\mcL}$ be a uniformly definite semipositive normed line bundle over $X$, very ample over $\Spec \ZZ$. Let $\mcF$ be a coherent sheaf over $X$. Then $$-\infty < \underline{c}_{d}(X,\overline{\mcL}, \mcF) = \overline{c}_{d}(X,\overline{\mcL},\mcF)<+\infty$$
Furthermore, if $\overline{\mcL}$ is ample, this invariant is positive.
\end{thm}

Let's first prove that the same reductions are possible:

\begin{lemma}\label{Arithmensembledesreductions} Let $X$ be a projective scheme over $\Spec \ZZ$ of dimension $d$. Let $\overline{\mcL}$ be a semipositive seminormed line bundle over $X$, $\mcF$ a coherent sheaf on $X$. Let $r\geq d$.
\begin{itemize}
    \item[1.] If the ideal sheaf $\mcI$ define the underlying reduced space $X_{red}$, and $\mcI^s= 0$, then, the arithmetic Hilbert invariants of $(X,\overline{\mcL},\mcF)$ are the sum of the arithmetic Hilbert invariants of $(X_{red}, \overline{\mcL}_{|X_{red}}, \mcI^j\mcF/\mcI^{j+1}\mcF)$ for $0\leq j \leq s-1$, as long as $\underline{c}_{r}(X_{red}, \overline{\mcL}_{|X_{red}}, \mcI^j\mcF/\mcI^{j+1}\mcF)=\overline{c}_{r}(X_{red}, \overline{\mcL}_{|X_{red}}, \mcI^j\mcF/\mcI^{j+1}\mcF)$.
    \item[2.] If $\underline{c}_{r}(X,\overline{\mcL},\mcO_X)=\overline{c}_{r}(X,\overline{\mcL},\mcO_X)<+\infty$, the arithmetic Hilbert invariants of $(X,\overline{\mcL},\mcF)$ are an alterned sum of arithmetic Hilbert invariants associated to $(X,\overline{\mcL},\mcO_X)$.
    \item[3.] The arithmetic Hilbert invariants of $(X,\overline{\mcL},\mcF)$ are a sum of arithmetic Hilbert invariant of triples of the form $(V,\overline{\mcL}_{|V}, \mcI)$ where $V$ is an integral closed subscheme of $X$, and $\mcI$ is an ideal sheaf on $V$,  as long as $\underline{c}_{r}(V,\overline{\mcL}_{|V}, \mcI)=\overline{c}_{r}(V,\overline{\mcL}_{|V}, \mcI)$ for this triples.
\end{itemize}
\end{lemma}

\begin{proof}
For point $1$, as there is a filtration $0 = \mcI^s\mcF \subset \mcI^{s-1}\mcF \subset \cdots \subset \mcI\mcF \subset \mcF$ and that $\mcI^j \mcF /\mcI^{j+1}\mcF$ are supported on
$X_{red}$, this is a consequence of the additivity with respect to exact sequences of coherent sheaves property and the projection formula given in proposition \ref{uniquenessinvariant}. \\

For point $2$, by the Hilbert Syzygy theorem, there is a finite resolution of $\mcF$ by coherent sheaves of the forms $\bigoplus_i \mcL^{\otimes l_i}$.\\
Note that the arithmetic Hilbert invariants associated to $(X,\overline{\mcL}, \mcL^{\otimes l_i})$ and to $(X,\overline{\mcL}, \mcO_X)$ are the same. Indeed the injective morphism of coherent sheaves $p_X^*\mcL^{l_i} =\bigoplus_{n\in \NN}  \mcL^{\otimes n + l_i} \to \bigoplus_{n\in \NN} \mcL^{\otimes n} = \mcO_\VV $ over $\VV_X(\mcL)$ induces a strict morphism on the spaces of sections.\\
Furthermore, the arithmetic Hilbert invariants associated to $(X,\overline{\mcL}, \mcO_X)$ are then finite, by assumption and the proposition \ref{uniquenessinvariant}.\\
Now, as $\overline{\mcL}$ is semipositive, by finiteness and the additivity with respect to exact sequences, this concludes for point $2$.

For point 3, there is a filtration $0= \mcF_0 \subset \mcF_1 \subset \dots \mcF_r = \mcF$ by coherent subsheaves such that for $1\leq j \leq r$, there is an integral subscheme $\iota_j : V_j \to X$ and an ideal sheaf $\mcI_j \subset \mcO_{V_j}$ such that $\mcF_j/\mcF_{j+1} \simeq {\iota_j}_* \mcI_j$, see    
\cite[\href{https://stacks.math.columbia.edu/tag/01YF}{Lemma 01YF}]{stacks-project}. Thus, by the additivity with respect to short exact sequences and the projection formula, this concludes for point $3$.\\
\end{proof}

\begin{proof}[Proof of theorem \ref{thprincipalbis}] Firstly, by point 2 of lemma \ref{Arithmensembledesreductions}, we are reduced to prove the case $\mcF = \mcO_X$.\\
Let $\overline{\mcL} = (\VV_X(\mcL), T)$.\\
By applying the deformation to the projective completion of the cone repetitively by keeping track of closed immersions to $\PP^N_k$, using corollary \ref{CombinaisonActionGmU1}, proposition \ref{invarianceArithmeticHilbert} and point $1$ and $2$ of lemma \ref{Arithmensembledesreductions}, we can assume that there is a closed immersion induced by sections of $\mcL$, $X\to \PP^N_\ZZ = \Proj \ZZ[x_0, \dots, x_N]$ such that, for all $0\leq i \leq N$, $X$ is stable by the action of $\GG_m$ by multiplication on $x_i$ and $T$ is stable by the corresponding action of $U(1)^{l+1}$. More precisely, this is possible because of the conservation of the arithmetic Hilbert invariants through the deformation to the projective completion of the cone which is possible for uniformly definite norms and reduced spaces, see proposition \ref{invarianceArithmeticHilbert}. The proposition \ref{invarianceArithmeticHilbert} also ensures that the seminorms considered are uniformly definite, and using proposition \ref{passageaureduit} and points 1 and 2 of lemma \ref{Arithmensembledesreductions} ensures that we can work on reduced spaces. Furthermore, by the geometric Hilbert-Samuel theorem and lemma \ref{invariancepardeformation}, the dimension is invariant by this deformation.\\
By noting then that the surjection $\ZZ[x_0, \dots, x_N] \to \bigoplus H^0(X,\mcL^{\otimes n})$ is graded in each variable $x_i$, the kernel is an ideal generated by monomes $n_{\underline{i}}\underline{x}^{\underline{i}}$, the irreducible components of $X$ are thickening of $\PP^l_\ZZ = \Proj \ZZ[x_{i_0}, \dots , x_{i_l}]$, or vertical components of the form $\PP^l_{\FF_p}$.\\
By point $3$ of \ref{Arithmensembledesreductions}, we are reduced to prove the theorem for the integral subschemes $V_i$ of $X$ and $\mcF$ is an ideal sheaf of $\mcO_{V_i}$. By point $2$ of \ref{Arithmensembledesreductions} again, we can assume that $\mcF = \mcO_{V_i}$.\\
Then, three cases arise, either $\dim V_i < \dim X$ and we conclude by induction, and $\overline{c}_{d}(V_i,\overline{\mcL}_{|V_i}, \mcO_{V_i})=\underline{c}_{d}(V_i,\overline{\mcL}_{|V_i}, \mcO_{V_i})=0$, either $\dim V_i = \dim X$ and $V_i$ is a vertical component and then the result is a direct corollary of the geometric Hilbert-Samuel theorem, either $\dim V_i = \dim X$ and $V_i$ is $\PP^l_\ZZ = \Proj \ZZ[x_{i_0}, \dots , x_{i_l}] $ with $l = \dim X-1$, where the analytic tube is stable by the actions by $U(1)$ by multiplication on $x_{i_0}, \dots, x_{i_l}$.\\
Finally, by using point 2 of \ref{Arithmensembledesreductions}, we are reduced to the case $(\PP^l_\ZZ=\Proj \ZZ[x_{i_0}, \dots , x_{i_l}], \overline{\mcO_\PP(1)}, \mcO_\PP)$, where the analytic tube is stable by the action of $U(1)^{l+1}$ by multiplication on $x_{i_0}, \dots, x_{i_l}$. And, in this case the theorem is true by proposition \ref{CasDePn} and the finiteness result in corollary \ref{finitenesslimit}.\\
\end{proof}

As a corollary, we obtain theorem \ref{thprincipal}:

\begin{cor}Let $X$ be a projective scheme over $\Spec \ZZ$ and $\overline{\mcL}$ a uniformly definite semipositive normed line bundle over $X$. Let $\mcF$ be a coherent sheaf over $X$. Let $d$ be the dimension of the support of $\mcF$.  Then, 
$$ - \infty <\underline{c}_{d}(X,\overline{\mcL}, \mcF) = \overline{c}_{d}(X,\overline{\mcL},\mcF)<+\infty$$
\end{cor}

\begin{proof}
Firstly, by the projection formula in proposition \ref{uniquenessinvariant}, by pulling back $\overline{\mcL}$ to the support of $\mcF$, we can assume that $d$ is the dimension of $X$.
Then, by point 3 and 2 of lemma \ref{Arithmensembledesreductions}, we can assume that $X$ is integral and that $\mcF = \mcO_X$.\\
Take $n_0$ such that $\mcL^{\otimes n_0}$ is very ample.\\
Then, for $i \in \NN$, by theorem \ref{thprincipalbis}, $$-\infty < \underline{c}_{d}(X,\overline{\mcL}^{\otimes n_0}, \mcL^{\otimes i}) = \overline{c}_{d}(X,\overline{\mcL}^{\otimes n_0},\mcL^{\otimes i})<+\infty$$
Let $q_X: \VV_X(\mcL^{\otimes n_0}) \to X$ be the total space of $\mcL^{\otimes -n_0}$ and $p_X: \VV_X(\mcL) \to X$ be the total space of $\mcL^\vee$. \\
Then, we have a finite map over $X$, $f :  \VV_X(\mcL) \to \VV_X(\mcL^{\otimes n_0})$. \\
Thus, the morphisms $$H^0(\VV_X(\overline{\mcL}^{\otimes n_0}), q_X^*\mcL^{\otimes i})\to H^0(\VV_X(\overline{\mcL}^{\otimes n_0}), f_*p_X^*\mcL^{\otimes i}) \to H^0(\VV_X(\overline{\mcL}), p_X^*\mcL^{\otimes i}) $$ are strict, by flat base change and by \cite{Bost-Charles}. This is the injection of $\bigoplus_n H^0(X,\mcL^{\otimes nn_0 + i})$ into $\bigoplus_{n} H^0(X,\mcL^{n+i})$. Therefore, the invariants $\overline{c}_{d}(X,\overline{\mcL}^{\otimes n_0},\mcL^{\otimes i})$ and $\underline{c}_{d}(X,\overline{\mcL}^{\otimes n_0},\mcL^{\otimes i})$ are given by taking subsequences of the sequences defining $\overline{c}_{d}(X,\overline{\mcL}, \mcO_X)$ and $\underline{c}_{d}(X,\overline{\mcL}, \mcO_X)$.\\

Now, let $s: \mcO_X \to \mcL^{\otimes kn_0+1}$ be a global section of $ \mcL^{\otimes kn_0+1}$. By tensoring by $\mcL^{\otimes i}$, we get a short exact sequence $0\to \mcL^{\otimes i} \to \mcL^{\otimes kn_0 + i +1 } \to \mcL^{\otimes i}_{|div(s)}\to 0  $.\\
By using the additivity with respect to short exact sequences, the fact that the support of $\mcL^{\otimes i}_{|div(s)}$ is of dimension strictly less than $X$ and theorem \ref{thprincipalbis}, we get the following equalities: 
 $$\overline{c}_{d}(X,\overline{\mcL}^{\otimes n_0},\mcL^{\otimes i})= \overline{c}_{d}(X,\overline{\mcL}^{\otimes n_0},\mcL^{\otimes kn_0+ i+1}) = \overline{c}_{d}(X,\overline{\mcL}^{\otimes n_0},\mcL^{\otimes i+1})$$
$$\underline{c}_{d}(X,\overline{\mcL}^{\otimes n_0},\mcL^{\otimes i})= \underline{c}_{d}(X,\overline{\mcL}^{\otimes n_0},\mcL^{\otimes kn_0+ i+1})= \underline{c}_{d}(X,\overline{\mcL}^{\otimes n_0},\mcL^{\otimes i+1})$$
And, hence, by equalities of the limits of the subsequences, it proves the equality :   
$$- \infty < \underline{c}_{d}(X,\overline{\mcL}, \mcO_X) = \overline{c}_{d}(X,\overline{\mcL},\mcO_X)<+\infty$$
This concludes.\end{proof}

In the non semipositive case, we still have a convergence result: 

\begin{cor}Let $X$ be a projective scheme over $\Spec \ZZ$ of dimension $d$ and $\overline{\mcL}$ a uniformly definite normed line bundle over $X$ such that the underlying line bundle $\mcL$ is ample over $\Spec \ZZ$, then 
$$-\infty < \underline{c}_{d}(X,\overline{\mcL}, \mcO_X) = \overline{c}_{d}(X,\overline{\mcL},\mcO_X)<+\infty$$
Furthermore, if $\overline{\mcL}$ is ample, this invariant is positive.
\end{cor}

\begin{proof} By proposition \ref{passageAlenveloppeholo}, as $\overline{c}_{d}(X,\overline{\mcL}, \mcO_X)$ and $\underline{c}_{d}(X,\overline{\mcL}, \mcO_X)$ are invariants of the arithmetic Hilbertian $\mcO_{\Spec \ZZ}$-module $H^0(\VV_X(\overline{\mcL}),\mcO_\VV)$, we can assume that $\mcL$ is endowed with the associated equilibrium norm. That is, we can assume that $\overline{\mcL}$ is a uniformly definite semipositive normed line bundle. Then, the result is theorem \ref{thprincipal}.\\
\end{proof}

Finally, the classical formulation of the Hilbert-Samuel theorem is a direct consequence of theorem \ref{thprincipal}: 

\begin{cor} Let $X$ be a projective scheme over $\Spec \ZZ$ of smooth generic fiber. Let $\overline{\mcL}$ be a semipositive Hermitian line bundle over $X$.
Let $\overline{\mcF}$ be a Hermitian vector bundle over $X$. Fix a continuous Riemannian metric on $X(\CC)$, and denote $||\cdot||_{L^2}$ the associated $L^2$ norms on $H^0(X(\CC),\mcL^{\otimes n}\otimes \mcF)$. Let $d$ the Krull dimension of $X$.
Then, 
\begin{center}
    $\displaystyle{\frac{d!}{n^d}}\widehat{\chi}(H^0(X,\mcL^{\otimes n} \otimes \mcF), ||\cdot||_{L^2})$ converges towards a finite limit when $n\to \infty$.
\end{center}
\end{cor}

\begin{proof}
This is a direct consequence of proposition \ref{uniquenessinvariant} and theorem \ref{thprincipal}.\\
\end{proof}

%% file: 3-Bornological/statementOfTheHilbertSamuelTheorem.tex
\section{Arithmetic Hilbert invariants}

\subsection{Arithmetic degree for seminormed Hermitian coherent sheaves.} We use the usual notion of arithmetic degree to develop invariants for graded arithmetic Hilbertian $\mcO_{\Spec \ZZ}$-modules.

\begin{defn}
A \emph{seminormed Hermitian coherent sheaf} over $\Spec \ZZ$ is a pair $\overline{M} :=(M,||\cdot||)$ where $M$ is a module of finite type over $\ZZ$ and $||\cdot||$ is a Hilbertian seminorm over $M_{\RR} \colon= M\otimes_{\ZZ}\RR$.\\
The \emph{arithmetic degree} of $\overline{M}$ is defined by: 
\begin{align*}
    \widehat{\chi}(\overline{M}) &  := 0 & \text{ if } M=0,\\
    & := +\infty & \text{ else, if } ||\cdot|| \text{ is not a norm,}\\
    &:= -\log \text{covol}(\overline{M}) + \log \#M_{\text{tor}} & \text{ otherwise}
\end{align*}
where the covolume is computed using the unique translation invariant Radon measure $\mu_{\overline{M}}$ on $M_{\RR}$ that satisfies the following condition: for any orthonormal basis $(e_1 , \dots  , e_N )$ of $(M_\RR , ||\cdot|| )$,$$ \mu_{\overline{M}}\left(\sum [0,1]e_i \right) = 1$$

\end{defn}

\subsubsection*{Additivity with respect to exact sequence.} One of the most important property of the arithmetic degree is the additivity with respect to exact sequence. \\

\begin{lemma}\label{additivityclassical} Let $0 \to \overline{N} \to \overline{M} \to \overline{P} \to 0$ be an exact sequence of seminormed Hermitian coherent sheaves over $\Spec \ZZ$, i.e. the seminorm on $P_\RR$ is the quotient seminorm and the seminorm on $N_\RR$ is the induced seminorm. Then $$\widehat{\chi}(\overline{M}) = \widehat{\chi}(\overline{P}) + \widehat{\chi}(\overline{N})$$
Furthermore, if $\{0\} = F_0 \subset \dots \subset F_n = M$ is a filtration of $M$ by coherent sheaves over $\Spec \ZZ$. Let $\overline{F_i/F_{i+1}}$  be the subquotient of $\overline{M}$, defined by $\overline{F_i/F_{i+1}}= (F_i/F_{i+1}, ||\cdot||_{sq})$ . Then, we have $$\widehat{\chi}(\overline{M}) = \sum \widehat{\chi}(\overline{F_i/F_{i+1}})$$
\end{lemma}

\begin{proof}
See, for instance, \cite{chambertloir}.
\end{proof}

\subsection{Arithmetic Hilbert invariants: definitions.}

Inspired by the arithmetic Hilbert-Samuel theorem, we define the following numerical invariants :

\begin{defn}\label{invariantforgradedmodule}
Let $A_\bullet$ be a graded algebra over $\ZZ$, finitely generated in degree $1$ of Krull. Let $M_\bullet$ be a graded arithmetic Hilbertian $\mcO_{\Spec \ZZ}$-module and assume furthermore that $M_\bullet$ is a finite graded $A_\bullet$-module. Let $d$ be the dimension of the support of $M_\bullet$. Let $r\geq d$.\\
Let $(||\cdot ||_i)_{i\in \mathbb{N}}$ be a decreasing family of Hilbertian seminorms invariant by complex conjugation defining the bornology on $M_\bullet$.
Let  $$\overline{c}_{r}(M_\bullet, (||\cdot||_i)) = \lim_{i}\limsup_n \frac{r!}{n^{r}}\widehat{\chi}(M_n, ||.||_i)$$
$$\underline{c}_{r}(M_\bullet, (||\cdot||_i)) = \lim_{i}\liminf_n \frac{r!}{n^{r}}\widehat{\chi}(M_n, ||.||_i)$$

\end{defn}

\paragraph{Notation.} This definition depends only on the bornology $\mcB$ on $M_\bullet$ (see \cite[lemma 5.6]{NI2022}) and this invariants are denoted $$\underline{c}_{r}(M_\bullet, \mcB) \text{ and } \overline{c}_{r}(M_\bullet, \mcB)$$

As explained in section \ref{partiebornologies}, if $X$ is a projective scheme over $\Spec \ZZ$, $\overline{\mcL}$ is a seminormed line bundle over $X$ which is ample over $\Spec \ZZ$, $\mcF$ is coherent sheaf over $X$ and $p_X:\VV_X(\mcL) \to X$ is the total space of $\mcL^{\vee}$, then $H^0(\VV_X(\mcL), p_X^*\mcF)$ inherits a structure of graded arithmetic Hilbertian $\mcO_{\Spec \ZZ}$-module, which is a finite module over the graded algebra $H^0(\VV_X(\mcL), \mcO_\VV)$.\\

\begin{defn}
With the above notation, if $\mcB$ is the bornology associated with the graded arithmetic Hilbertian $\mcO_{\Spec \ZZ}$-module $H^0(\VV_X(\mcL), p_X^*\mcF)$ and $r\geq \dim \mathrm{Supp} \mcF$ then we denote by 
\begin{center}
    $\overline{c}_{r}(X,\overline{\mcL}, \mcF) $ and $\underline{c}_{r}(X,\overline{\mcL}, \mcF)$
\end{center} the invariants $\overline{c}_{r}(H^0(\VV_X(\mcL), p_X^*\mcF), \mcB) $ and $\underline{c}_{r}(H^0(\VV_X(\mcL), p_X^*\mcF), \mcB)$, respectively.\\
 $\overline{c}_{r}(X,\overline{\mcL}, \mcF) $ and $\underline{c}_{r}(X,\overline{\mcL}, \mcF)$ are called the \emph{arithmetic Hilbert invariants associated to  $(X,\overline{\mcL},\mcF)$}.\\
\end{defn}

\subsection{Arithmetic Hilbert invariants: properties.}

Reflecting the bornological properties of the total space of sections of a hermitian line bundle, the arithmetic Hilbert invariants verify the following properties:

\begin{prop}\label{uniquenessinvariant}
The numerical invariants $\overline{c}_{r}(X,\overline{\mcL},\mcF)$ and $\underline{c}_{r}(X,\overline{\mcL},\mcF)$
where $X$ is a projective scheme over $\Spec \ZZ$ of dimension $d$, $\overline{\mcL}$ is a seminormed line bundle over $X$ ample over $\Spec \ZZ$, $\mcF$ is a coherent sheaf over $X$ and $r\geq d$ with $d= \dim \mathrm{Supp} \mcF$, verify the following properties: 
\begin{itemize}
    \item Lower finiteness: $\overline{c}_{r}(X,\overline{\mcL},\mcF)$ and $\underline{c}_{r}(X,\overline{\mcL},\mcF)$ belong to $\RR \cup \{+\infty\}$
    \item Positivity: if $\overline{\mcL}$ is ample, $\overline{c}_{r}(X,\overline{\mcL},\mcF)$ and $\underline{c}_{r}(X,\overline{\mcL},\mcF)$ are nonnegative. Furthermore, they are positive if $r= d$.
    \item Projection formula: if $i:Y\to X$ is a closed immersion, $\mcF$ is a coherent sheaf on $Y$, $\overline{\mcL}$ is a seminormed line bundle over $X$, then $$\overline{c}_{r}(X,\overline{\mcL},i_*\mcF) = \overline{c}_{r}(Y,i^*\overline{\mcL},\mcF)$$ $$\underline{c}_{r}(X,\overline{\mcL},i_*\mcF) = \underline{c}_{r}(Y,i^*\overline{\mcL},\mcF)$$
    \item Additivity with respect to exact sequences: if $\overline{\mcL}$ is semipositive and $0 \to \mcE \to \mcF \to \mcG \to 0$ is an exact sequence of coherent sheaves, then 
    $$\overline{c}_{r}(X,\overline{\mcL},\mcF) = \overline{c}_{r}(X,\overline{\mcL},\mcE) + \overline{c}_{r}(X,\overline{\mcL},\mcG)$$ 
    $$\underline{c}_{r}(X,\overline{\mcL},\mcF) = \underline{c}_{r}(X,\overline{\mcL},\mcE) + \underline{c}_{r}(X,\overline{\mcL},\mcG)$$
    if $\underline{c}_{r}(X,\overline{\mcL},\mcE) = \overline{c}_{r}(X,\overline{\mcL},\mcE)$ or $\underline{c}_{r}(X,\overline{\mcL},\mcG) = \overline{c}_{r}(X,\overline{\mcL},\mcG)$.
    \item Smooth case with continuous norms: if $X_\CC$ is smooth, and if $|\cdot|$ is a continuous norm on $\mcL$, $\mcF$ is vector bundle on $X$, then $\mcF$ can be endowed with a structure $\overline{\mcF}$ of Hermitian vector bundle over $X$ and if we fix a Riemannian continuous metric invariant by complex conjugation on $X(\CC)$, then
    $$\overline{c}_{r}(X,\overline{\mcL}, \mcF) = \limsup_n \frac{r!}{n^{r}}\widehat{\chi}(H^0(X,\mcL^{\otimes n}\otimes \mcF), ||.||_{L^2})$$
$$\underline{c}_{r}(X,\overline{\mcL}, \mcF) = \liminf_n \frac{r!}{n^{r}}\widehat{\chi}(H^0(X,\mcL^{\otimes n}\otimes \mcF), ||.||_{L^2})$$
    \item Monotonicity: if $|\cdot| \geq |\cdot|'$ are two seminorms on $\mcL$. Then:
    $$\overline{c}_{r}(X,(\mcL,|\cdot|), \mcF)\leq  \overline{c}_{r}(X,(\mcL,|\cdot|'), \mcF)$$  $$\underline{c}_{r}(X,(\mcL,|\cdot|), \mcF)\leq  \underline{c}_{r}(X,(\mcL,|\cdot|'), \mcF)$$
\end{itemize}
\end{prop}

\begin{proof}
See \cite[section 5]{NI2022}.
\end{proof}

%% file: 4-Deformation/introDeformation.tex
Following the key idea for intersection theory as developed in \cite{Baum1975RiemannrochFS} and \cite{fulton}, we use an explicit deformation to compute our invariants. The aim is a geometric interpretation of the "dévissage" technique.
This deformation is the deformation to the projective completion of the cone and was introduced in its full generality in \cite{NI2022}. 
It takes a projective scheme $X$ over $\Spec \ZZ$, a seminormed line bundle $\overline{\mcL}$ such that the underlying line bundle is very ample over $\Spec \ZZ$, a closed immersion $i:Y\to X$ which is a hyperplane section with respect to $\mcL$, and associate a projective scheme $D_YX$ over $\Spec \ZZ$, projective over $\PP^1_\ZZ$, and a semipositive seminormed line bundle $\overline{\mcL'}$. $D_YX$ contains the deformation to the normal cone as an open subscheme, see proposition 8.1 in  \cite{NI2022}. We will use here the special case where $Y$ is the divisor associated to a global section of a very ample line bundle.\\
We will describe here the main properties of the deformation to the projective completion of the cone, see \cite{NI2022} for further details.

%% file: 4-Deformation/algebraicdeformation.tex
\subsection{Main geometric properties.}

Let $A$ be either a field $k$ or $\ZZ$. \\
The data is the following :
\begin{itemize}
    \item A projective scheme $X$ over $\Spec A$.
    \item A closed immersion $i : Y \to X$.
    \item A line bundle $\mcL$ over $X$ very ample over $\Spec A$, such that the ideal $I_\bullet \subset \bigoplus H^0(X,\mcL^{\otimes n})$ is generated in degree 1.
\end{itemize}
\paragraph{Terminology:} We say that $Y$ is a \emph{hyperplane section with respect to $\mcL$}.\\

Then the deformation to the projective completion of the cone associated to $(X,Y,\mcL)$ is denoted $(D_YX,\mcL')$ and is a projective scheme over $\Spec A$ and a projective scheme over $\PP^1_A$.\\
In the case where $Y=\mathrm{div}(s)$ with $s\in H^0(X,\mcL)$, the notation is $(D_sX, \mcL')$.\\

It verifies the following properties :

\begin{prop}\label{associatedamplelinebundle} Let $X$ be a projective scheme over $\Spec A$. Let $\mcL$ be a line bundle over $X$ very ample over $\Spec A$. Let $i:Y\to X$ be a hyperplane section with respect to $\mcL$.\\
Let $(D_YX, \mcL')$ be the deformation to the projective completion of the cone associated to $(X,Y,\mcL)$. Then,
\begin{itemize}
    \item the line bundle $\mcL'$ over $D_YX$ is ample over $\Spec A$,
    \item the restriction of $(D_YX,\mcL')$ over $1 \in \PP^1_A$ is $(X,\mcL)$,
    \item the restriction of $(D_YX,\mcL')$ over $\infty \in \PP^1_A$ is $(\Proj \bigoplus I_\bullet^l/I_\bullet^{l+1}, \mcO(1))$ where 
   $\mcO(1)$ is the canonical line bundle associated to the generated in degree $1$ graded algebra $\bigoplus I_\bullet^l/I_\bullet^{l+1}$.
\end{itemize}
\end{prop}

\begin{proof}
See \cite[proposition 3.14]{NI2022}.
\end{proof}

\begin{rem} In particular, $(\PP^k_A,\mathrm{div}(X_0),\mcO_\PP(1))$ is deformed  via the deformation to the projective completion of the cone to $(\PP^k_A,\mcO_\PP(1))$.\\
\end{rem}

\begin{prop}\label{immerpn} Let $X$ be a projective scheme over $\Spec A$. Let $\pi: X\to \PP^{N}_A= \Proj A[x_0,\dots ,x_N]$ be a closed immersion. Let $\mcL = \pi^* \mcO_\PP(1)$. $s= \pi^*x_N \in H^0(X,\mcL)$. \\
Let $(D_sX,\mcL')$ be the deformation to the projective completion of the cone associated to $(X,\mathrm{div}(s),\mcL)$.\\
Then, there is a canonical closed immersion above $\PP^1_A$ from $D_sX$ to the closed subscheme of $\Proj A[x_0,...,x_N,y] \times_A \Proj A[u,t]$ defined by the equation $ty = ux_N$. \\
In particular, this describes a closed immersion $\pi':D_sX_{\infty} \to \Proj A[x_0,...,x_{N-1},y]$ and $\mcL'_{\infty} = \pi'^*\mcO_\PP(1)$
\end{prop}

\begin{proof}
See \cite[proposition 3.9 and proposition 3.16]{NI2022}.
\end{proof}

\subsection{The action of $\GG_m$.}
Another important property is the existence of an action of $\GG_m$ on $(D_YX,\mcL')$, that is an equivariant action on $\VV_{D_YX}(\mcL')\to D_YX \to \PP^1_A$ for the natural action of $\GG_m$ on $\PP^1_A$. This action is called \emph{the natural transverse action}.\\
By the invariance of $\infty$ by the natural action of $\GG_m$, the fiber above $\infty$ of $\VV_{D_YX}(\mcL')$ and $ D_YX$ is stable by the action of $\GG_m$.\\

Here is a description of this action on the fiber above $\infty$.
\begin{prop}
Let $X$ be a projective scheme over $\Spec A$. Let $\pi: X\to \PP^{N}_A= \Proj A[x_0,\dots ,x_N]$ be a closed immersion. Let $\mcL = \pi^* \mcO_\PP(1)$. $s= \pi^*x_N \in H^0(X,\mcL)$. \\
Let $(D_sX,\mcL')$ be the deformation to the projective completion of the cone associated to $(X,\mathrm{div}(s),\mcL)$.\\
Let $\pi':D_sX_{\infty} \to \Proj A[x_0,...,x_{N-1},y]$ be the closed immersion given in proposition \ref{immerpn} and $\mcL'_{\infty} = \pi'^*\mcO_\PP(1)$.\\
Then the action of $\GG_m$ on $(D_sX_\infty, \mcL'_{\infty})$ is the restriction of the action of $\GG_m$ on $(\Proj A[x_0,...,x_{N-1},y], \mcO_\PP(1))$ by multiplication on the last variable.
\end{prop}

\begin{proof}
See \cite[proposition 3.17]{NI2022}.
\end{proof}

Furthermore, the decomposition of the total space of section of $\mcL'$ into its isotypic components under the natural transverse action of $\GG_m$ explains that this construction can be seen as a geometric interpretation of the "dévissage" technique. See the proposition 3.20 of \cite{NI2022} for further details. \\

As a direct corollary of this decomposition, we have the invariance of the Hilbert functions through the deformation: 

\begin{cor}\label{invariancepardeformation}
Let $X$ be a projective scheme over a field $k$. Let $\mcL$ be a line bundle over $X$ very ample over $\Spec k$. Let $i:Y\to X$ be a hyperplane section with respect to $\mcL$.\\
Let $(D_YX, \mcL')$ be the deformation to the projective completion of the cone associated to $(X,Y,\mcL)$.\\
Then, $(X,\mcL)$ and $(D_YX_\infty , \mcL')$ have the same Hilbert functions. That is, for $n\gg 0$, $$\dim_k H^0(X,\mcL^{\otimes n}) = \dim_k H^0(D_YX_\infty,\mcL'^{\otimes n}) $$
\end{cor}

\begin{proof}
See \cite[proposition 3.20]{NI2022}.
\end{proof}

\subsection{Compatibility with other $\GG_m$-actions.} Let consider deformations to the projective completion of the cone associated to a global section of $\mcL$. Thanks to the proposition \ref{immerpn}, we can keep track of closed immersions to $\PP^N_A$. We deform  $(\pi:X\to \PP^N_A, \mathrm{div}(s), \mcL=\pi^*\mcO_\PP(1))$ into $(\pi':D_sX_{\infty}\to \PP^N_A, \pi'^*\mcO_\PP(1))$. One the key feature is that the natural transverse action of $\GG_m$ on $(D_sX, \mcL')$ adds an action of $\GG_m$ and that $D_sX_\infty$ inherits the existing ones on $(X,\mcL)$:\\

\begin{prop}\label{CombinaisonActionGm}
Let $X$ be a projective scheme over $\Spec A$, $\mcL$ be a line bundle over $X$ very ample over $\Spec A$.\\
Let $\pi :X \to \PP^N_A = \Proj A[x_0, \dots , x_N]$ be a closed immersion induced by global sections of $\mcL$.
Let $s = \pi^*x_N$.\\
Let $(D_sX, \mcL')$ be the deformation to the projective completion of the cone associated to $(X,\mathrm{div}(s),\mcL)$.\\
Assume that $\GG_m$ acts on $\VV_{\PP^N_A}(\mcO_\PP(1))$ by multiplication on the first variable $x_0$ of $\PP^{N}_A$ and that this action restricts to $\VV_{X}(\mcL)$. \\
Then, this action naturally extends to $(D_sX,\mcL')$.\\
Moreover, $D_sX_\infty$ is stable by this action, and if $\pi':D_sX_{\infty} \to \Proj A[x_0,...,x_{N-1},y]$ is the closed immersion given in proposition \ref{immerpn} and $\mcL'_{\infty} = \pi'^*\mcO_\PP(1)$. Then the corresponding action of $\GG_m$ on $(D_sX_\infty, \mcL'_{\infty})$ is the restriction of the action of $\GG_m$ on $(\Proj A[x_0,...,x_{N-1},y], \mcO_\PP(1))$ by multiplication on the first variable.
\end{prop}

\begin{proof}
See \cite[corollary 3.23]{NI2022}.
\end{proof}

Bearing in mind the convenient proposition \ref{reducedcaseBornology}, some actions of $\GG_m$ carries out to the underlying reduced subspace.

\begin{prop}\label{passageaureduit}
Let $X$ be a projective scheme over $\Spec A$. Let $\mcL$ be a line bundle over $X$ very ample over $\Spec A$. Assume that the global sections of $\mcL$ induce a closed immersion $\pi :X \to \PP^N_A = \Proj A[x_0, \dots , x_N]$.\\
Suppose that the action of $\GG_m$ on $\PP^N_A$ by multiplication on $x_0$ restricts to $X$. Then this action restricts to the underlying reduced space $X_{red}$.
\end{prop}

\begin{proof}
Consider the restriction to $X$ map: $A[x_0, \dots, x_N] \to \bigoplus H^0(X,\mcL^{\otimes n})$. This is a graded map for the graduation by the total degree of polynomials and for the graduation of $x_0$. The kernel of this map is a graded ideal for those two graduation and the radical of a graded  ideal is still graded. This concludes.\\
\end{proof}

%% file: 4-Deformation/deformationOfAnalyticsTubes.tex
\subsection{The deformation of analytic tubes.}

The second main feature of this deformation is that we can associate to it a deformation of analytic tubes. 
More precisely, the data is the following :

\begin{itemize}
    \item A projective scheme $X$ over $\Spec \ZZ$.
    \item A seminormed line bundle $\overline{\mcL}$ over $X$ such that the underlying line bundle $\mcL$ is very ample over $\Spec \ZZ$.
    \item  $i : Y \to X$ a hyperplane section with respect to $\mcL$.
\end{itemize}

Then the deformation to the projective completion of the cone associated to $(X,Y,\overline{\mcL})$ is denoted $(D_YX, \overline{\mcL'})$, where $(D_YX, \mcL')$ is the deformation to the projective completion of the cone associated to $(X,Y,\mcL)$ and $\mcL'$ is endowed with a structure of seminormed line bundle.\\

It verifies the following properties :

\begin{prop}
Let $X$ be a projective scheme over $\Spec \ZZ$. Let $\overline{\mcL}$ be a semipositive seminormed line bundle over $X$, very ample over $\Spec \ZZ$. Let $i:Y\to X$ be a hyperplane section with respect to $\mcL$. \\
Let $(D_YX,\overline{\mcL'})$ the deformation to the projective completion of the cone associated to $(X,Y,\overline{\mcL})$.
\begin{itemize}
    \item $\overline{\mcL'}$ is semipositive.
    \item the fiber over $1$ of $(D_YX, \overline{\mcL'})$ is $(X,\overline{\mcL})$.
\end{itemize}
\end{prop}

\begin{proof}
See \cite[definition 3.25 and proposition 3.26]{NI2022}.\\
\end{proof}

\subsection{Actions of $U(1)$ on the analytic tube.} One key feature of this definition of the seminormed structure on $\mcL'$ is that the analytic tube on $\mcL'$ is compatible with the action of $\GG_m$, in the sense that it is stable by the natural transverse action of $U(1)$, the unit complex numbers. For further details, see \cite[proposition 3.28]{NI2022}.\\

Now, as in proposition \ref{CombinaisonActionGm}, some actions of $(\GG_m,U(1))$ on $(X,\overline{\mcL})$ extends to $(D_sX,\overline{\mcL'})$ and restricts to $(D_sX_\infty,\overline{\mcL'_\infty})$.

\begin{cor}\label{CombinaisonActionGmU1}
Let $X$ be a projective scheme over $\Spec \ZZ$. Let $\overline{\mcL} = (\VV_X(\mcL), T)$ be a seminormed line bundle on $X$ such that the underlying line bundle is very ample over $\Spec \ZZ$.\\
Let $\pi :X \to \PP^N_\ZZ = \Proj \ZZ[x_0, \dots , x_N]$ be a closed immersion induced by the global sections of $\mcL$.
Let $s = \pi^*x_N$.\\
Let $(D_sX, \overline{\mcL'}) = (D_sX, (\VV_{D_sX}(\mcL'), T'))$ be the deformation to the projective completion of the cone associated to $(X, \mathrm{div}(s),\overline{\mcL})$.\\
Assume that $\GG_m$ acts on $\VV_{\PP^N_\ZZ}(\mcO_\PP(1))$ by multiplication on the first variable $x_0$ of $\PP^{N}_\ZZ$, that this action restricts to $\VV_{X}(\mcL)$ and that $T$ is stable by $U(1)$. \\
Then, this action naturally extends to $\VV_{D_sX}(\mcL')$ and $T'$ is stable by the restriction of this action to $U(1)$.\\
At the fiber over infinity, there is a closed immersion $\pi' : D_sX_\infty \to \PP^N_\ZZ = \Proj \ZZ[x_0,\dots,x_{N-1},y]$ equivariant for the action of $\GG_m$ by the multiplication on the first variable. \\
Moreover, the natural transverse action of $\GG_m$ on $D_sX$ at the fiber at infinity is the restriction through $\pi'$ of the multiplication on last variable of $\PP^N_\ZZ$.\\
Furthermore, if $\overline{\mcL'_\infty} = (\VV_{D_sX_\infty}(\mcL'_\infty), T'')$ is the restriction of $\overline{\mcL'}$ to the fiber at infinity, then $T''$ is stable by the restriction of this two actions to $U(1)$.
\end{cor}

\begin{proof}
See \cite[corollary 3.31]{NI2022}.\\
\end{proof}

%% file: 4-Deformation/invarianceHfunctionByDeformation.tex
\subsection{Conservation of the arithmetic Hilbert-Samuel invariants by the deformation to the projective completion of the cone.} The deformation to the projective completion of the cone give a deformation from $(X,\overline{\mcL})$ to $(D_YX_\infty, \overline{\mcL'})$. As an analog of the conservation of the Hilbert functions in corollary \ref{invariancepardeformation}, we have the conservation of the arithmetic Hilbert invariants:

\begin{prop}\label{invarianceArithmeticHilbert}Let $X$ be a reduced projective scheme over $\Spec \ZZ$. Let $\overline{\mcL}=(\VV_X(\mcL),T)$ be a semipositive seminormed line bundle over $X$, very ample over $\Spec \ZZ$. Assume furthermore that the metric is uniformly definite on $\overline{\mcL}$. Let $i:Y\to X$ be a hyperplane intersection with respect to $\mcL$.\\
Let $(D_YX, \overline{\mcL'})$ be the deformation to the projective completion of the cone associated to $(X,Y,\overline{\mcL})$.\\
Then, 
\begin{itemize}
    \item $X$ and $D_YX_{\infty}$ have the same dimension $d$.
    \item $\overline{\mcL}_{|D_YX_\infty}$ is uniformly definite.
    \item $(X, \overline{\mcL}, \mcO_X)$ and $(D_YX_{\infty}, \overline{\mcL'},\mcO_{D_YX_{\infty}} )$ have the same arithmetic Hilbert invariants.
\end{itemize}
\end{prop}

\begin{proof}
See \cite[theorem 6.3 and proposition 3.32]{NI2022}.\\
\end{proof}